\documentclass[11pt,a4paper]{article}
\usepackage[applemac]{inputenc}

\usepackage[english]{babel}
\usepackage[a4paper,left=2.5cm,right=4cm,top=2.5cm,bottom=2.5cm]{geometry}
\usepackage{microtype}
\usepackage{xurl}
\usepackage{amsmath}
\usepackage{stmaryrd,xcolor}
\usepackage[normalem]{ulem}
\usepackage{color}
\usepackage{float} 
\usepackage{amsfonts}
\usepackage{amssymb,amsthm}
\usepackage{hyperref} 
\usepackage{cleveref} 
\usepackage{graphicx,abstract}
\usepackage{mathrsfs}

\def\m{\mu}

\newcommand{\extr}{{\rm Ext}}
\newcommand{\exti}{{\rm Ext}_i}
\newcommand{\extun}{{\rm Ext}_1}

\newcommand{\vall}{\mathrm{Val}}
\newcommand{\vol}{\mathrm{vol}}

\newcommand{\NNN}{{\mathcal N}}
\newcommand{\FFF}{{\mathcal F}}
\newcommand{\M}{{\mathbb M}}

\newcommand{\Sph}{S(0,1)}

\newcommand{\radproj}{\rho\,} % radial projection symbol

\newcommand{\ds}{\displaystyle}

\newcommand{\tal}{\mathrm{Tal}}

\newcommand{\T}{{\mathrm{talw\,}}}
\newcommand{\crit}{\mathrm{crit}}
\newcommand{\argminloc}{\mathrm{argmin}\text{-}\mathrm{loc}\,}

\newcommand{\tr}{\mbox{{\rm tr}}}

\numberwithin{equation}{section}
\newtheorem{theorem}{Theorem}[section]
\newtheorem{proposition}[theorem]{Proposition}
\newtheorem{corollary}[theorem]{Corollary}
\newtheorem{lemma}[theorem]{Lemma}

\theoremstyle{definition}
\newtheorem{remark}[theorem]{Remark}
\newtheorem{definition}[theorem]{Definition}

\newtheorem{assumption}[theorem]{Assumption}

\newtheorem{claim}{Claim}

\newenvironment{proof*}{\noindent{\bf Proof.}}{\qed}

\title{Gradient extremals, talwegs, valleys, and directional alignment for generic gradient descent} 

\author{%
Pascal B\'egout\thanks{%
Toulouse School of Economics, Universit\'e Toulouse Capitole,
Institut de Math\'ematiques de Toulouse, Toulouse,
France, \href{mailto:Pascal.Begout@math.cnrs.fr}{Pascal.Begout@math.cnrs.fr}, \href{https://orcid.org/0000-0002-9172-3057}{https://orcid.org/0000-0002-9172-3057}.}%
\and
J\'er\^ome Bolte\thanks{%
Toulouse School of Economics, Universit\'e Toulouse Capitole, \'Ecole Normale Sup\'erieure, France, \href{mailto:jbolte@ut-capitole.fr}{jbolte@ut-capitole.fr}.}
\and
Thomas Mariotti\thanks{%
Toulouse School of Economics, CNRS, University of Toulouse Capitole, Toulouse, France, CEPR, and CESifo. \href{mailto:thomas.mariotti@tse-fr.eu}{thomas.mariotti@tse-fr.eu}.}
\and
Francisco J. Silva\thanks{%
Universit\'e de Limoges, XLIM, UMR CNRS 7252, France, \href{mailto:francisco.silva@unilim.fr}{francisco.silva@unilim.fr}.}%
}

\def\dd{{\rm d}}

\def\weight(#1,#2){c_{#1,#2}}
 \if {
  
 } \fi
\if {

}{{\caption}}
} \fi

\def\D{\mathcal{D}}

\def\M{\mathcal{M}}

\def\T{\mathcal{T}}

\def\eps{\varepsilon}

\def\det{\mathop{\rm det}}

\def\dist{\mathop{\rm dist}}

\def\half{\mbox{$\frac{1}{2}$}}

\def\1B{{\bf  1}}

\newcommand{\NN}{\mathbb{N}}

\newcommand{\RR}{\mathbb{R}}

\newcommand\be{\begin{equation}}
\newcommand\ee{\end{equation}}
\newcommand\ba{\begin{array}}
\newcommand\ea{\end{array}}
\newcommand{\bean}{\begin{eqnarray*}}
\newcommand{\eean}{\end{eqnarray*}}

\def\ds{\displaystyle}

\begin{document}

\maketitle

\begin{abstract} 
Gradient extremals are loci along which the gradient is an eigenvector of the Hessian. These objects provide a natural geometric framework connecting several notions, notably valleys and talwegs, which we analyze from a variational viewpoint in the generic case. We then show that trajectories of the gradient flow and of its discrete counterpart exhibit directional alignment with the tangent spaces to gradient extremals, and generically to the talweg. Under non-resonance assumptions, and in contrast with the quadratic case, alignment rates are governed either by the first spectral gap or by the smallest eigenvalue of the Hessian at the limit point. Nonlinearities and the step length may both distort these rates in a complex manner. We further prove a volume concentration phenomenon emphasizing the structuring role of gradient extremals: for large times, the images of sets of initial conditions concentrate inside valleys and asymptotically around talwegs.
\end{abstract}
\maketitle

\section{Introduction}\label{introduction}

Cauchy's gradient descent method is central in many branches of mathematics and has gained renewed prominence because of its fundamental role in applications such as modern machine learning, see e.g. \cite{bach2024learning}, or optimal transport and PDEs, see e.g. \cite{santambrogio20151}. 

In this paper we investigate the directional properties of the gradient descent dynamics and of its discrete counterpart. This is not a new theme: it was, for instance, part of the program surrounding what was known as Thom's gradient conjecture, eventually established in \cite{kurdyka2000proof}. Our objective here is rather different as we aim to clarify several geometric structures that govern the behavior of trajectories near critical points in the generic case, that is, for functions having a nondegenerate Hessian with simple spectrum at their critical points.

To this end we formalize a number of natural objects that have appeared in various contexts \cite{Gelfand1961,d2005explicit}, sometimes in other fields (for instance in chemistry, see e.g. \cite{schlegel1992following}), but whose presentation in the literature remains scattered. We argue that these objects may be unified through the notion of gradient extremals. Gradient extremals are paths along which the gradient is an eigenvector of the Hessian, and they are not, in general, gradient curves. Equivalently, they may be interpreted as loci where the slope of the function is critical relative to its level sets, under suitable nondegeneracy assumptions. This framework naturally leads to the notions of valleys and crest extremals, and calls for a clarification of their connection with the concept of a talweg. We study these notions in the generic case and show in particular that the variational definition of a talweg ---minimal slope locus on each level set--- corresponds to the gradient extremal associated with the minimal eigenvalue \Cref{la_croix_theoreme}. We also investigate the infinitesimal structure of valleys, see \Cref{prop:tangent-valley}.

Talwegs have a long history in topography and mathematics, starting with the pioneering works of de Saint-Venant \cite{Saint-Venant}, Cayley \cite{Cayley}, Maxwell \cite{Maxwell}, and Boussinesq \cite{Boussinesq} on contour and slope lines, and the definition of hills and dales. More recently, talwegs have played a fundamental role in the theory of convergence, in particular through Lojasiewicz-type inequalities \cite{lojasiewicz1963propriete}. Their definition is not uniform across the literature; here we view them as paths along which the slope is locally minimal along the corresponding level set. The works of D'Acunto \cite{d2001courbes} and Kurdyka \cite{kurdyka1998gradients} are pioneering in the use of this concept and in establishing its connection with the convergence of gradient curves, although it also appears in Lojasiewicz work \cite{lojasiewicz1982trajectoires}. In \cite{bolte2010characterizations} a transparent link between the two notions is provided as it is shown that the existence of a talweg of finite length is both necessary and sufficient for the Kurdyka-Lojasiewicz inequality to hold.

Our analysis departs from these general situations and focuses on the generic case, where gradient trajectories are known to converge generically to local minimizers \cite{thom1949partition,pemantle1990nonconvergence,lee2016gradient}. In this stability framework, gradient extremals, talwegs, and crest extremals form locally one-dimensional manifolds that intersect orthogonally at their common critical point. They provide a geometric skeleton for the gradient flow, much as eigenspaces provide natural coordinates for gradient descent quadratic functions, i.e., for the exponential function of a symmetric matrix. 
In a second part of the paper we show that talwegs act as secondary attractors to gradient dynamics: generically, directions and secants of gradient trajectories, both in continuous and discrete time, asymptotically align with the tangent directions of the talwegs, see \Cref{sec:directional_convergence_continuous_case} and \Cref{sec:discret}. Alignment occurs exponentially in the first spectral gap and the first eigenvalue. Actually the phenomenon is not only generic as it happens for all initial conditions, and at many scales with alignment to other gradient extremals with convergence rates governed by spectral gaps, see \Cref{th:align} and \Cref{th:align_discrete}. We thus, partially, recover the behavior observed in the quadratic case with a convergence structured by eigenspaces. Let us observe, however, that the nonlinear case differs from the linearized one, as nonlinearities may disrupt the dominant effect of the first spectral gap when the conditioning along the talweg is too poor, i.e., when the smallest eigenvalue at a local minimizer is too small compared to the others; see \Cref{rem:alignement-rate-cont} and \Cref{rem:alignment_speed_sharpness_discrete}. Finally, we establish volume concentration phenomena within the valley: for sets of initial conditions of positive measure, discrete and continuous flow concentrates along talwegs as time grows, see Theorems \ref{rapport_de_volumes} and \ref{rapport_de_volumes_cas_discret}. 

The tools used in this work rely on optimization techniques together with classical results from the theory of dynamical systems around hyperbolic points, in particular in the case where the real eigenvalues have the same sign, see \cite{MR141856,MR96853}. To facilitate the reading easy or routine proofs are postponed to the Appendix.

%%%%%%%%%%%%%%%%%%%%%%%%%%%%%%%%%%%%%%%%%%%%%%%%%%%%%%%
\section{Gradient extremals, talweg, and valleys}\label{preliminaries}

\subsection{Preliminaries} 

We assume $d\geq 2$, $\RR^d$ is endowed with the Euclidean scalar product
$\langle \cdot,\cdot\rangle $ and associated norm $|\cdot|$. For a nonempty subset $F\subset\RR^d$, we write $\dist(x,F)=\inf\{|x-y|:y\in F\}$.

Given $f\colon\RR^d \to \RR$, for every $r\in\RR$, we denote by $[f=r]$, $[f\leq r]$, $[f<r]$ level and sublevel sets. When some point $a$ belongs to $[f< r]$, we denote by $[f< r]_a$ the connected component of $[f< r]$ containing $a$. If $f$ is $C^2$ and the Hessian at $x^*$ is positive definite, the family $[f<r]_{x^*}$ forms a family of connected open neighborhoods of $x^*$ for $r>r^*$ with $r^*:=f(x^*)$. For this, one can, for instance, observe that 
\begin{equation}
 f(x) \geq f(x^*)+c_f |x-x^*|^2, \mbox{ with }c_f>0,
\label{quadratic_growth}
\end{equation}
on a nonempty open ball centered at $x^*$. For $r>r^*$, one defines similarly $[f\leq r]_{x^*}$ and set $[f=r]_{x^*}:=\{x\in [f\leq r]_{x^*}: f(x)=r\}$.

\medskip

Let us consider the following generic situation \cite{golubitsky2012stable}:
\begin{assumption}[Blanket assumption]\label{ass:H} $f\colon \RR^d\to \RR$ is of class $C^{3}$ and $x^*\in\RR^{d}$ is a nondegenerate critical point with simple spectrum, i.e., 
\begin{equation}\label{eq:hess}
\nabla f(x^*)=0\text{ and } \nabla^2 f(x^*) \mbox{ has pairwise distinct nonzero eigenvalues},
\end{equation}
where $\nabla f$ and $\nabla^{2}f$ denote the gradient and the Hessian of $f$, respectively. 
\end{assumption}
When \Cref{ass:H} holds, we adopt some notations and convention (on $\eta, \lambda_i,v_i$, $i=1,\ldots,d$) that we now proceed to describe. Locally the spectrum of $\nabla^2 f$ varies smoothly near $x^*$ (see Corollary~\ref{cordiff} in the Appendix): there indeed exist $\eta>0$ and $C^1$ maps $\lambda_i\colon B(x^*,\eta)\to \RR$ and $v_i\colon B(x^*,\eta)\to \RR^d$ ($i=1,\dots,d$) such that, for every $x\in B(x^*,\eta)$, $\lambda_i(x)$ is a simple eigenvalue of $\nabla^2 f(x)$ with associated eigenvector $v_i(x)$ and the family $\{v_{i}\}_{i=1}^{d}$ is orthonormal. We denote by $P(x)$ the orthogonal matrix whose $i^{th}$ column is $v_i(x)$, and by $\mathscr{D}(x)$ the diagonal matrix whose $i^{th}$ diagonal entry is $\lambda_i(x)$. After relabeling, if necessary, we may assume
\be
\label{valeurs_propres_ordonnees}
\lambda_1(x)<\lambda_2(x)<\dots<\lambda_d(x),\quad\text{for all }x \mbox{ in }B(x^*,\eta),
\ee
and thus, by the inverse function theorem, we may further assume that 
\begin{equation}
\nabla f(x)\neq 0,\quad\text{for all }x\in B(x^*,\eta)\setminus\{x^*\}.
\label{eq:non_null_gradient}
\end{equation}

We shall also consider the assumption
\begin{assumption}[Strong local minimizer]
\label{ass:min} 
For $x^*$ as in \Cref{ass:H}, we have $\lambda_{1}>0$ throughout $B(x^*,\eta)$.
\end{assumption}

Under~\Cref{ass:min}, $x^*$ is a strict local minimum satisfying~\eqref{quadratic_growth} and there exists $\bar r\in (r^*,\infty)$ such that $[f\leq \bar r]_{x^*}\subset B(x^*,\eta)$.

%%%%%%%%%%%%%%%%%%%%%%%%%%%%%%%%%%%%%%%%%%%%%%%%%%%%%%%
\subsection{Gradient extremals}

For a $C^1$ function, {\em gradient extremals} are rectifiable curves $\gamma(r)\in E(r)$ selected in the multivalued mapping:
$$r\rightrightarrows E(r):=\left\{x \in[f=r]: \nabla f(x)\neq 0 \mbox{ and }x \mbox{ is a critical point of }|\nabla f| \text{ over }[f=r] \right\}$$
where $r$ ranges over the values of $f$. But under \Cref{ass:H}, which is a local assumption, and using Lagrange multipliers conditions, it is more convenient to adopt here the following:
\begin{definition}[Gradient extremals]\label{gradient_extremal} Under \Cref{ass:H}, for $i=1, \hdots,d$, the {\em $i^{th}$ gradient extremal} is the set 
$$
\exti= \left\{ x \in B(x^*,\eta) \::\:\nabla^2f(x)\nabla f(x)= \lambda_i(x) \nabla f(x) \right\}.
$$ 
\end{definition}

\begin{theorem}[Gradient extremals are one dimensional smooth manifolds]\label{la_croix_theoreme}
Under \Cref{ass:H}, for all $i=1, \hdots, d$, the gradient extremal $\exti$ is a one-dimensional $C^1$ embedded submanifold passing through $x^*$, its tangent space at $x^*$ is given by $T_{x^*}\exti:=\RR v_i(x^*).$
\end{theorem}

\begin{proof*}
 Defining $B(x^*,\eta)\ni x\mapsto z(x)=P^{\top}(x) \nabla f(x)\in\RR^d$, Definition~\ref{gradient_extremal} implies that, for every $i=1, \hdots, d$ and $x\in B(x^*,\eta)$, 
\begin{equation}
x\in \exti\quad \Leftrightarrow \quad \mathscr{D}(x) z(x)= \lambda_i(x) z(x) \quad \Leftrightarrow \quad  z_j(x)=0 \mbox{ for all }j\neq i.
\label{eq:caracterisation_talweg}
\end{equation}

Notice that $z$ is of class $C^{1}$ with 
\begin{equation}
Dz(x^*)=P^\top(x^*)\nabla^2 f(x^*)+DP^\top(x^*)\nabla f(x^*)=P^\top(x^*)\nabla^2 f(x^*)
\label{eq:dz}
\end{equation}
invertible. Therefore, $z$ is a local diffeomorphism. Let us show the result for $i=1$: let $\pi$ be the canonical projection on the last $d-1$ coordinates. Define $S\colon B(x^*,\eta)\ni x\mapsto\pi(z(x))\in \RR^{d-1}$. By~\eqref{eq:caracterisation_talweg}, we have that $\extun=S^{-1}(\{0\})\ni x^*$. Moreover, $S$ is a submersion. Indeed, $DS(x)=\pi(Dz(x))$ and $z$ is a local diffeomorphism. As a conclusion, $\extun$ is a one-dimensional $C^1$ embedded submanifold with $T_{x^*}\extun=\ker DS(x^*)$. Since, by~\eqref{eq:dz},
\begin{gather*}
DS(x^*)=\pi\left(P^\top(x^*)\nabla^2 f(x^*)\right)=\pi\left(\mathscr{D}(x^*)P^\top(x^*)\right),
\end{gather*}
the result follows from $\ker DS(x^*)=\text{span}\{v_{2}(x^*),\hdots,v_{d}(x^*)\}^{\perp}=\RR v_{1}(x^*)$.
\medskip
\end{proof*}

%%%%%%%%%%%%%%%%%%%%%%%%%%%%%%%%%%%%%%%%%%%%%%%%%%%%%%%
\subsection{Talweg and crest extremal} 
 
Let us recall the notion of a {\it talweg} around $x^*$,\footnote{We restrict to the case when $x^*$ is a local minimizer.} see, e.g.,  \cite{d2001courbes,d2005explicit,bolte2010characterizations} and references therein.

\begin{definition}[Talweg: ``path in the valley"\footnote{English translation of the German word borrowed from topography: Talweg or Thalweg.}]\label{definition_talweg}  
Under Assumptions \ref{ass:H} and \ref{ass:min}, the  {\it talweg} of $f$ around $x^*$ is the set-valued map $\tal: [r^*,\bar r)\rightrightarrows [f<\bar r]_{x^*}$ defined through: 
\be
\theta(r)\in \tal(r)\Longleftrightarrow \theta(r) \; \mbox{is a local minimizer of  }|\nabla f|\mbox{ over } [f=r]_{x^*}.
\label{talweg_problem}
\ee
\end{definition}
Since $[f=r]_{x^*}$ is compact, we have $\tal(r)\neq \emptyset$. 

\begin{theorem}[The talweg is the gradient extremal of minimal eigenvalue] Under Assumptions \ref{ass:H} and \ref{ass:min}, 
\be
\label{egalite_talweg_extremal_1} 
\tal\big([r^*,\bar r)\big)= \extun\cap [f<\bar r]_{x^*},
\ee
for $\bar r$ close enough to $r^*.$ 
In particular, $\tal\big([r^*,\bar r)\big)$ is a one-dimensional $C^1$ embedded submanifold containing $x^*$ and
\be
\label{espace_tangent_intersection}
T_{x^*} \left[\tal\big([r^*,\bar r)\big)\right]= \RR v_1(x^*).
\ee
\label{th:talweg_is_gradient_extremal}
\end{theorem}
\begin{proof*}  Let $x\in \tal\big([r^*,\bar r)\big)$. If $f(x)=r^*$, then $x=x^*$, whence $x\in \extun\cap [f<r_{\eta}]$. Thus, let us assume that $r:=f(x)\in (r^*,\bar r)$. The Lagrangian $L_r: [f<\bar r]_{x^*}\times \RR \to \RR$ associated to the optimization problem in  \eqref{talweg_problem}  is given by
$$
L_r(y,\lambda)= \half |\nabla f(y)|^2 +\lambda(r-f(y))$$ 
for all $y\in [f<\bar r]_{x^*},\;\lambda\in\RR$. Since $\nabla f(x)\neq 0$, there exists a unique $\bar{\lambda}(x) \in \RR$ such that $\nabla_x L_r(x,\bar{\lambda}(x)) =0$, or, equivalently,
\be\label{lagrange_multiplier_talweg_definition_}
 \nabla^2f(x) \nabla f(x) = \bar{\lambda}(x)  \nabla f(x).
\ee
Moreover, the following second-order necessary condition holds (see, e.g., \cite[Chapter 2, Theorem 2]{MR1058438}):
\be
\langle \nabla^2_{xx} L_r(x, \bar{\lambda}(x))h, h \rangle \geq 0, \hspace{0.3cm} \text{for all }  h \in \RR^d \; \text{such that }  \langle \nabla f(x),  h\rangle =0, 
\label{second_order_condition_necessary_talweg}
\ee
which can be rewritten as 
\be
\left\langle \left( D^3f(x) \nabla f(x) + \nabla^2f(x)^2-\bar{\lambda}(x) \nabla^2f(x) \right)h, h \right\rangle \geq 0,
\label{second_order_necessary_condition_third_derivative_f}
\ee
for all $h \in \RR^d$ such that  $\langle \nabla f(x), h\rangle =0$. If $\bar{\lambda}(x)\neq\lambda_{1}(x)$, we can take $h=v_{1}(x)$ in~\eqref{second_order_necessary_condition_third_derivative_f}. Indeed $v_1$ has to be orthogonal to $\nabla f(x)$ which is also an eigenvector. This  yields  
\begin{equation}
\left\langle   D^3f(x) \nabla f(x)v_{1}(x),v_{1}(x) \right\rangle + \lambda_{1}(x)\left( \lambda_{1}(x)-\bar{\lambda}(x)\right)\geq 0.
\label{eq:inequality_tressian}
\end{equation}
Since the first term tends to zero when $x$ approaches $x^*$ and $\bar \lambda(x)\in \{\lambda_{2}(x),\ldots, \lambda_d(x)\}$, the continuity of the spectrum of $\nabla^{2}f$ at $x^*$ shows that~\eqref{eq:inequality_tressian} can not hold for $|x-x^*|$ small enough. Thus, reducing $\bar r-r^*$, if necessary, we deduce that $\bar{\lambda}(x)=\lambda_{1}(x)$ and, hence, $x\in \extun\cap [f<\bar r]_{x^*}$.

Conversely, let $x \in \extun\cap [f<\bar r]_{x^*}$ and $h\in \RR^d$ be such that $\langle \nabla f(x), h \rangle=0$. Since $\nabla f(x)$ is in $\RR v_1(x)$,   there exist $\mu_2, \hdots, \mu_d\in \RR$ such that $h=\sum_{i=2}^{d} \mu_i v_{i}(x)$. Thus, 
\be\label{prop_3_1_estimates}
\begin{split}
&\;\left\langle \left(\nabla^2f(x)^2-\lambda_{1}(x)\nabla^2f(x) \right)h, h \right\rangle \\
=&\;\ds \sum_{i=2}^{d}\left(\lambda_i^2(x)  -\lambda_{1}(x)\lambda_i(x)\right) \mu_i^2\\ 
\geq &\;\ds \lambda_2(x) \left(\lambda_2(x) -\lambda_{1}(x)\right) \sum_{i=2}^{d}\mu_i^2 \\
= &\; \lambda_2(x) \left(\lambda_2(x) -\lambda_{1}(x)\right)|h|^2.
\end{split}
\ee
On the other hand, since $D^3f$, $\nabla f$, and $\lambda_2-\lambda_1$ are continuous and $\nabla f(x^*)=0$, reducing $\bar r-r^*$, if necessary, there exists $c>0$ such that 
$$
\lambda_2(x) \left(\lambda_2(x) -\lambda_{1}(x)\right)   - \left\| D^3f(x) \right\| |\nabla f(x)| >c.$$
Then we deduce from \eqref{prop_3_1_estimates} that 
\be
\left\langle \left( D^3f(x) \nabla f(x) + \left(\nabla^2f(x)\right)^2-\lambda_{1}(x)\nabla^2f(x) \right)h, h \right\rangle \geq    c|h|^2.
\label{second_order_condition_sufficient_talweg_proof}
\ee
Thus, setting $r=f(x)$, we get
\begin{equation}
\nabla_{x}L_{r}(x, \lambda_{1}(x))=0 \; \; \mbox{and } \;  \langle \nabla^2_{xx}L_{r}(x, \lambda_{1}(x))h,h\rangle >0,
\label{second_order_condition_sufficient_talweg_proof_lag}
\end{equation}
for all $h\in \RR^d\setminus \{0\}$ such that $\langle \nabla f(x),h \rangle=0$, 
which  by \cite[Theorem 4]{MR1058438} implies that $x\in \tal\big([r^*,\bar r)\big)$. Altogether, we have established~\eqref{egalite_talweg_extremal_1}. Finally, since $[f<\bar r]_{x^*}$ is an open subset of $\RR^d$, we deduce \eqref{espace_tangent_intersection} from~\eqref{egalite_talweg_extremal_1} and Theorem~\ref{la_croix_theoreme}.
\medskip
\end{proof*}

\begin{remark}
(a) (Talweg structure) As  shown in \Cref{th:study_of_theta} (see Appendix), if $r>r ^*$ is close enough to $r^*$, $\tal(r)$ has exactly two points, while $\tal(r^*)=\{x^*\}$.\\
(b) (Crest extremal)   One could also define a notion of a  {\em crest extremal} $\mathrm{Cre}\colon[r^*,\bar r)\rightrightarrows \RR^d$  through  
\be
\label{crest extremal_problem}
\gamma(r)\in \mathrm{Cre}(r) \Longleftrightarrow \gamma(r) \; \mbox{is a local maximizer of  }|\nabla f|\mbox{ over } [f=r]_{x^*}.  \ee
 Using similar arguments one can prove that the crest extremal is associated to the gradient extremal of maximal eigenvalue $\extr_d$.  
\end{remark}
\begin{figure}[ht]
    \centering
    \begin{minipage}{0.47\textwidth}
        \centering
        \includegraphics[width=\linewidth]{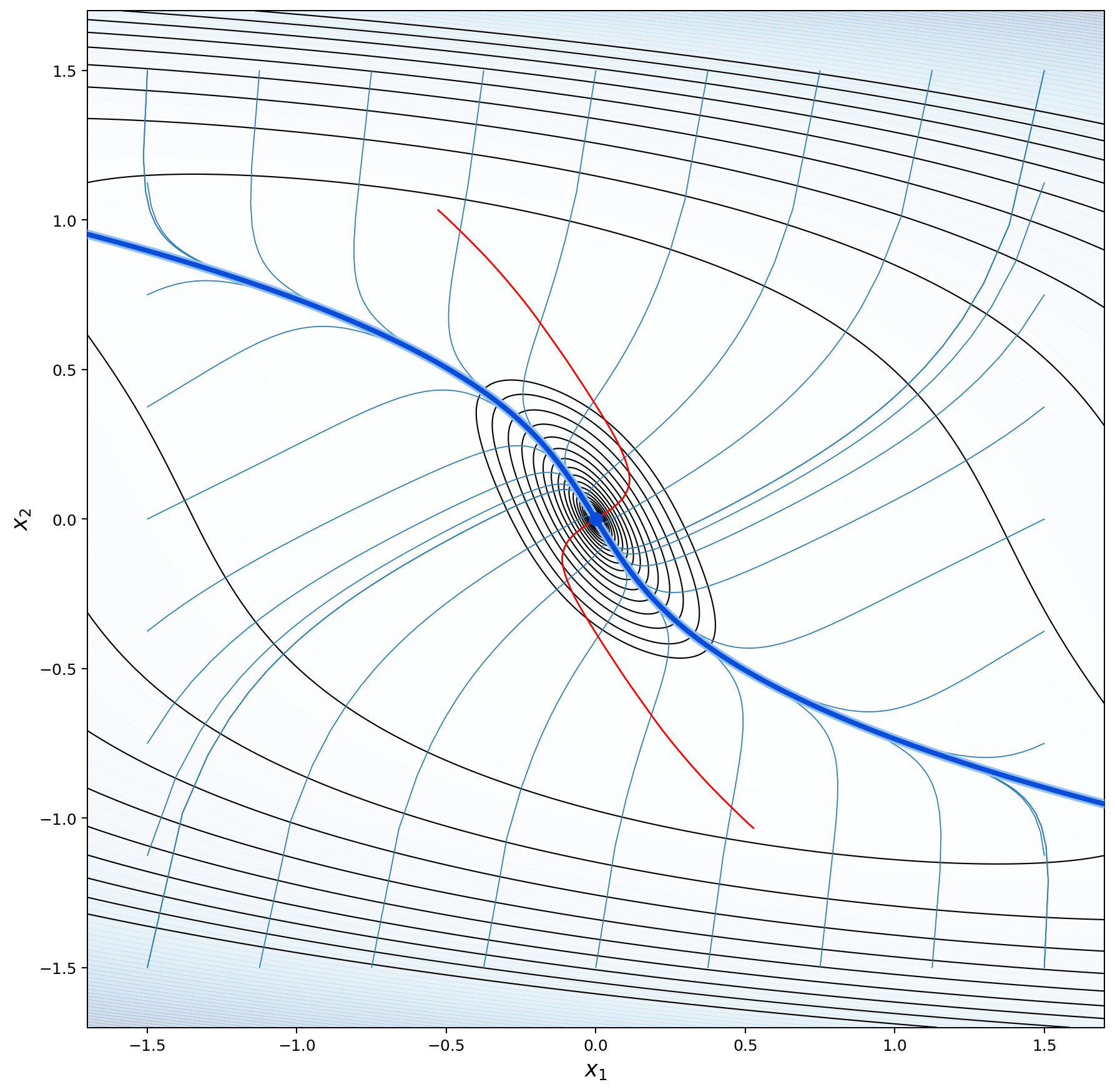}
        \caption{Black level lines, blue talweg, and red crest extremal for $f(x_1,x_2)=x_1^2+\left(x_1+x_2+x_2^3\right)^2$. Gradient curves, in light blue, ultimately align with the talweg, see \Cref{HG_cas_continu} and \Cref{th:align}.}
    \end{minipage}
    \hfill
    \begin{minipage}{0.47\textwidth}
        \centering
        \includegraphics[width=\linewidth]{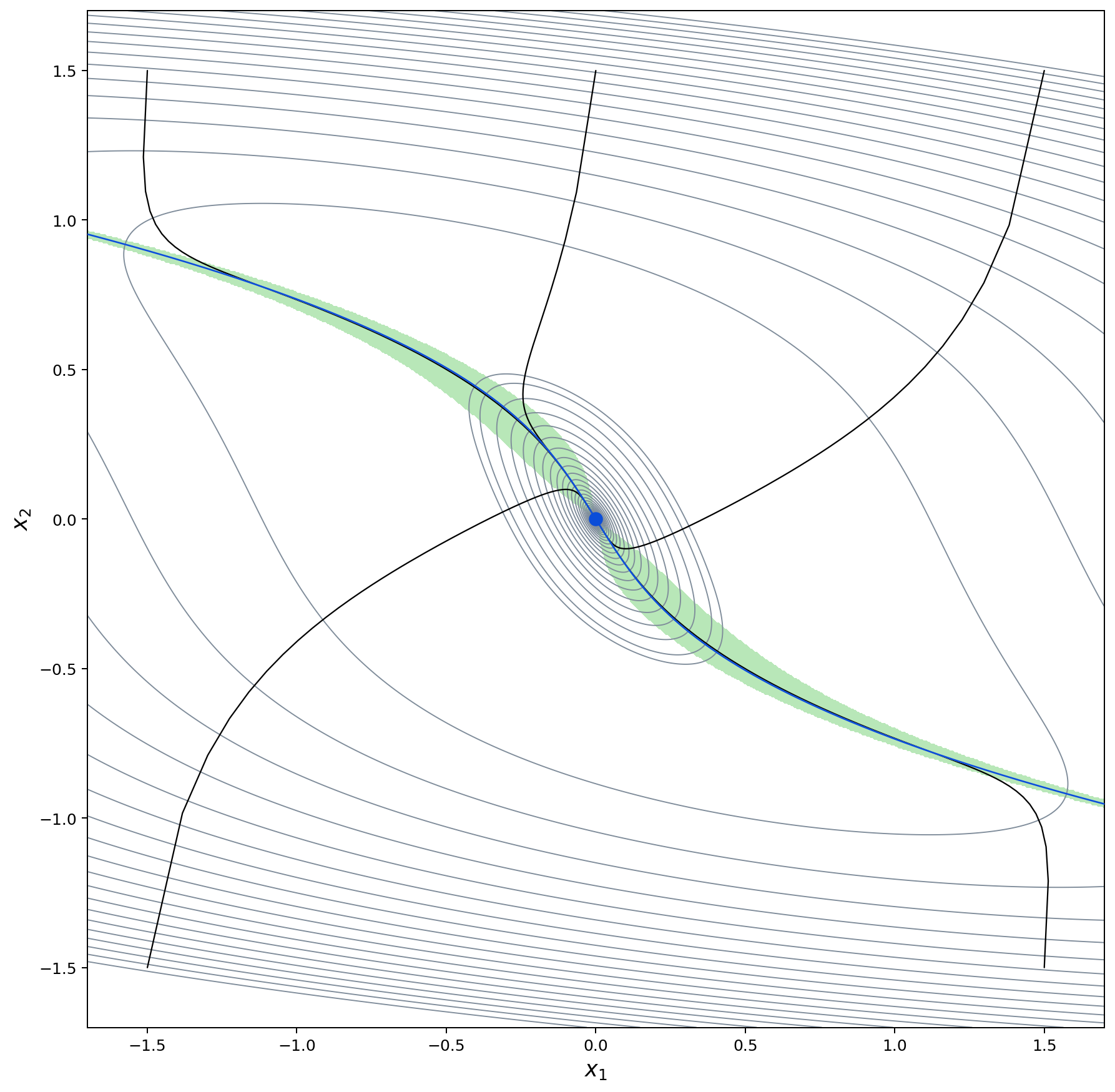}
        \caption{The valley has a simple infinitesimal structure near a local minimizer as it has the structure of a second-order cone, see \Cref{valley_inclusions}. Farther away, it may become more complex.}
    \end{minipage}
\end{figure}

%%%%%%%%%%%%%%%%%%%%%%%%%%%%%%%%%%%%%%%%%%%%%%%%%%%%%%%
\subsection{Valleys and their variational structure}
\label{section_valleys}

Under Assumptions \ref{ass:H} and \ref{ass:min}, and given $w\geq 0$,  define the {\it valley of width $w$} around $x^*$ by
\begin{equation}\label{def:val}
\vall_{f,w}(x^*)=\left\{ x\in [f<\bar r]_{x^*} \; : \; |\nabla^2f(x) \nabla f(x) - \lambda_{1}(x)  \nabla f(x) |  \leq  w | \nabla f(x)| \right\}.
\end{equation}
The ``width measure" $w$ evaluates the relative error of being an eigenvector for the Hessian and quantifies a form of enlargement of the talweg. Other measures could be proposed as in \cite{bolte2010characterizations}, but as we will see, this one has interesting properties.

By Theorem~\ref{th:talweg_is_gradient_extremal}, we have $\tal\big([r^*,\bar r)\big)\subseteq \vall_{f,w}(x^*)$ for all $w\geq 0$, with equality when $w=0$. Note also that when the width is such that $w\geq \lambda_d(x)-\lambda_1(x)$ for all $x$ in $[f<\bar r]_{x^*}$, a simple computation shows that the valley is so wide that it coincides with the whole neighborhood $[f<\bar r]_{x^*}$.

Consider the Taylor expansion at order two of $f-f(x^*)$ at $x^*$, i.e. $q(y)= f(x^*)+\half \langle \nabla^2f(x^*)y,y\rangle$, with $y$ in $\RR^d$, 
and introduce its valley (at $0$)
$$
\ba{rcl}
\vall_{q,w}   &=&  \left\{ x\in \RR^d \; : \; | \big(\nabla^2f(x^*)^2   - \lambda_{1}(x^*) \nabla^2f(x^*)\big)x |  \leq  w |\nabla^2f(x^*) x|  \right\}.
\ea
$$

One easily sees that $\vall_{q,w}$ is a non-empty  closed  cone whose generator is the eigenspace $\RR v_1(x^*)$. For every $0<w_1<w_2$, one has $\vall_{q,w_1}\subset\vall_{q,w_2}$.

\begin{proposition}[Valley and  conic approximation]
\label{valley_inclusions} 
Under Assumptions~\ref{ass:H}, \ref{ass:min}, and given $w>w'>0$, the following  hold:\smallskip\\
{\rm(i)}  Diminishing $\bar r>r^*$ if necessary, we have  $\big(x^*+\vall_{q,w'}\big)\cap [f<\bar r]_{x^*} \subset \vall_{f,w}(x^*)$.   \\[4pt]
{\rm(ii)} Diminishing $\bar r>r^*$ if necessary, we have  $\vall_{f,w'}\subset \big(x^*+\vall_{q,w}\big) \cap [f<\bar r]_{x^*}$.
\end{proposition}

\begin{proof*}  The proof of  {\rm(ii)} being analogous to the proof of {\rm(i)}, we only show the first assertion. Since  $f$ is $C^3$,  $\nabla f(x^*)=0$, and $\nabla^2f(x^*)$ are invertible, by Taylor expansions we have the existence of $c_1$, $c_2$ in $(0,\infty)$ such that, if $\delta>0$ is small enough and $x$ in $B(x^*,\delta)$, $|x-x^*|\leq c_1|\nabla f(x)|$ and $|\nabla f(x) - \nabla^2f(x^*)(x-x^*)|\leq c_2|x-x^*|^{2}$. Thus, setting $c=c_1c_2>0$, we have
$$
|\nabla f(x) - \nabla^2f(x^*)(x-x^*)| \leq c |x-x^*|  |\nabla f(x)|,\quad\text{for all } x\in B(x^*,\delta).
$$
The continuity of $\nabla^2 f$ and  $\lambda_1$ implies then the existence of an increasing and continuous function $m\colon\RR_+\to \RR_+$, with $m(0)=0$, such that for any $x\in B(x^*,\delta),$
\begin{equation}
\begin{split}
& \;\left|\left|\nabla^2f(x) \nabla f(x) - \lambda_{1}(x)  \nabla f(x) \right|- \left| \big(\nabla^2f(x^*)   - \lambda_{1}(x^*)I_{d}\big)\nabla^2f(x^*)(x-x^*) \right|\right|\\
\leq & \; m(|x-x^*|)|\nabla f(x)|.
\label{eq:consequence_taylor}
\end{split}
\end{equation}

Let $0<w'<w$ and $0<\delta<\eta$ be given and let $x\in x^*+\vall_{q,w'}\cap B(0,\delta)$, i.e., $|x-x^*|<\delta$ and
\begin{equation}
|\big(\nabla^2f(x^*)-\lambda_{1}(x^*)I_{d}\big)\nabla^2f(x^*)(x-x^*)|\leq w'|\nabla^2f(x^*)(x-x^*)|.
\label{eq:valley_application}
\end{equation}
By reducing  $\delta$, if necessary, we can assume that $m(\delta)\leq w-w'$. Thus,~\eqref{eq:consequence_taylor} yields 
$$
|\nabla^2f(x) \nabla f(x) - \lambda_{1}(x)  \nabla f(x) |  \leq  (w'+m(|x-x^*|))| \nabla f(x)|\leq w|\nabla f(x)|,
$$
i.e., $x\in \vall_{f,w}$. Diminishing $\bar r$ so that $[f<\bar r]_{x^*}\subset B(x^*,\delta)$ gives (i).
\end{proof*}

\medskip

Recall that given a non-empty closed set $A \subseteq \RR^d$, the {\em(Bouligand)  tangent cone} $\T_{A}(x)$ to $A$ at $x\in A$ is defined by
\begin{multline*}
\T_{A}(x)=   \big\{ h \in \RR^d \; : \; \exists \, (h_n)_{n \in \NN} \subset \RR^d, \; (\tau_n)_{n \in \NN} \subset (0,\infty), \\\tau_n \to 0, \; h_n \to h, \, \; \mbox{and } x+\tau_n h_n \in A,\hspace{0.2cm} \text{for} \; n \in \NN\big\}. 
\end{multline*}
Notice that, for every $w>0$, $\T_{\vall_{q}^{w}}(0)=\vall_{q,w}$.

\begin{proposition}[Tangent cone to the valley] \label{prop:tangent-valley} Under Assumptions~\ref{ass:H}, \ref{ass:min}, and given $w>0$, one has  $$\T_{\vall_{f,w}(x^*)}(x^*)=\vall_{q,w}.$$ 
\end{proposition}
\begin{proof*} Fix $w'>w$ and let $h$ in $\T_{\vall_{f,w}(x^*)}(x^*)$,  $(h_{n})_{n\in \NN}\subset \RR^d$ and $(\tau_{n})_{n\in \NN}$ a positive sequence such that $x^*+\tau_n h_n \in \vall_{f,w}(x^*)$, $h_n \to h$ and $\tau_n \to 0$. Since $\tau_n h_n \to 0$, Proposition \ref{valley_inclusions}{\rm(ii)} implies the existence of  $\delta\in (0,\eta)$ such that, for $n \in \NN$ large enough, $\tau_n h_n \in \vall_{q,w'}$ and, hence, $h_n \in  \vall_{q,w'}$. Since $\vall_{q,w'}$ is closed, we get that $h\in \vall_{q,w'}$ and, since $w'>w$ is arbitrary, we obtain that $h \in \vall_{q,w}$. Now, let $h\in \vall_{q,w}$ and take two sequences $(w_n)_{n\in \NN}\subset [0,w)$ and $(h_n)_{n \in \NN}\subset \RR^d$ such that, as $n\to \infty$,  $w_n \uparrow w$, $h_{n} \to h$ and  $h_n \in \vall_{q,w_n},$ for all $n\in \NN$.   Proposition \ref{valley_inclusions}{\rm(i)} then implies the existence of $(\delta_{n})_{n\in \NN} \subset (0, \eta)$ such that $x^*+\vall_{q,w_n} \cap B(0, \delta_n) \subset \vall_{f,w}(x^*)$. Let $(t_n)_{n\in \NN}\subset (0,\infty)$ be such that $t_n \downarrow 0$ and $t_n h_{n} \in \vall_{q,w_n} \cap B(0, \delta_n),$ for all $n\in \NN$, then $x^*+t_n h_{n} \in \vall_{f,w}(x^*)$ and hence $h\in \T_{\vall_{f,w}(x^*)}(x^*)$.
\end{proof*}

%%%%%%%%%%%%%%%%%%%%%%%%%%%%%%%%%%%%%%%%%%%%%%%%%%%%%%%
\section{Directional convergence of gradient trajectories}
\label{sec:directional_convergence_continuous_case}

\subsection{Preliminaries}
\paragraph{Radial projection} We  make repeated use of {\em the radial projection}
\[
\radproj:\RR^d\setminus\{0\}\longrightarrow \Sph,\qquad
\radproj(x)=\frac{x}{|x|}.
\]
where $S(0,1)$ is the unit sphere. For later use, let us record the following elementary properties of $\rho$. First, we have $\rho(\ell x) = \rho(x)$ for all $\ell > 0$ and $x \neq 0$. Moreover:
\begin{equation}
|\rho(a+b)-\rho(a)|\leq 2\frac{|b|}{|a|},\quad\text{for all }a,\, b\in\RR^{d},\, a\neq 0,\, a+b\neq 0.
\label{eq:inegalite_basique_normalisee}
\end{equation}
 Observe indeed $\displaystyle 
\frac{a+b}{|a+b|}-\frac{a}{|a|}
= \frac{a+b}{|a+b|}-\frac{a+b}{|a|}
+\frac{b}{|a|}
= (a+b)\!\left(\frac{1}{|a+b|}-\frac{1}{|a|}\right)
+\frac{b}{|a|},$  and thus,
$$\displaystyle
\left| \frac{a+b}{|a+b|}-\frac{a}{|a|}\right|
\le |a+b|\frac{\left||a|-|a+b|\right|}{|a+b||a|}
+\frac{|b|}{|a|}
\le 2\frac{|b|}{|a|}.
$$

\paragraph{Gradient trajectories} 

\smallskip

In view of our regularity assumptions one considers the {\it gradient flow} $ \Phi_t:\RR^{d}\to\RR^d$, $t\geq 0$, of $f$, which is defined through 
$$
\begin{aligned}
\partial_t \Phi_t(x)&=-\nabla f\left(\Phi_t(x)\right)\quad\text{for }t\in\RR,\\
\Phi_0(x)&=x, \quad\text{for }x\in\RR^d.
\end{aligned}
$$ 
If $f$ is coercive, and generic in the sense of $C^2$ topology \cite[Theorem 6.2]{golubitsky2012stable}, i.e. $f$ is a Morse function,  then any critical point $x^*$ of $f$ satisfies~\Cref{ass:H}. In that case gradient curves generically converge with respect to initial conditions towards local minimizers, see Theorem~\ref{th:align}~(i). The latter seems to be due, in essence, to \cite{thom1949partition}; it was rediscovered several times, e.g., by \cite{pemantle1990nonconvergence, goudou2009gradient,lee2016gradient} under different forms and settings.

\smallskip

%%%%%%%%%%%%%%%%%%%%%%%%%%%%%%%%%%%%%%%%%%%%%%%%%%%%%%%
\subsection{Directional convergence: the generic case}

In this section we focus our attention to local minimizers. 

Based on some fundamental results by Hartman (see \cite{MR141856} and Chapter 9 in \cite{MR1929104}), we establish in Theorem~\ref{HG_cas_continu}  below that, generically in terms of the initial condition $x$ near $x^*$, the asymptotic direction of the trajectory $(\Phi_t(x))_{t\geq 0}$  belongs to the eigenspace associated with the smallest eigenvalue of $\nabla^2f(x^*)$. \Cref{th:align} shows that this phenomenon is actually generic over $\RR^d$ under coercivity assumptions, and that for sufficiently smooth functions generic convergence rates may be provided. 

As in Section~\ref{preliminaries}, consider a local parametrization $\{(\lambda_i(x),v_i(x))\,:\,i=1,\hdots, d\}$ of eigenvalues and eigenvectors of $\nabla^{2}f(x)$, for $x$ near $x^*$, satisfying \Cref{ass:min}. Let us introduce a useful partition of $\RR^{d}$. 
Let ${\mathcal S}_{0}=\{0\}$ and, for every $1\le i\le d$, set
\begin{equation} 
\label{lem:partition}
{\mathcal S}_i=\left\{\sum_{j=d-i+1}^d \alpha_j v_j(x^*)\,:\,\alpha_{j}\in\RR\;\;\;\text{for all $j=d-i+1,\hdots,d\;$ with }\;\alpha_{d-i+1}\neq 0\right\}.
\end{equation}
Note that each ${\mathcal S}_i$ is a manifold of dimension $i$ and $\{{\mathcal S}_{i}\}_{i=0}^{d}$ forms a partition of $\RR^d$. Moreover, ${\mathcal S}_d$ is dense and the radial projection $\radproj$ maps ${\mathcal S}_d$ onto a dense open subset of the unit sphere. 

For $(t,y)$ in $[0, \infty)\times \RR^d$  set
\begin{equation}
 \Psi_t(y)= \exp\left(-\nabla^2f(x^{*}) t\right)
\label{def:Psi}
\end{equation}
the flow associated with the linearized system $y \mapsto -\nabla^2f(x^*)y$ on $\RR^d$. Note that for every $i=1,\hdots,d$ and $y=\sum_{j=d-i+1}^d \alpha_j v_j(x^*)\in {\mathcal S}_{i}$ we have 
\begin{equation}
\forall\,t\in [0,\infty),\quad \Psi_t(y)=\sum_{j=d-i+1}^d \alpha_j\exp(-\lambda_j(x^*) t) v_{j}(x^*),
\label{eq:Psi_t_explicit}
\end{equation}
and hence, by~\eqref{valeurs_propres_ordonnees},
\begin{equation}
\forall\,t\in [0,\infty), \quad
 |\Psi_t(y)|\leq |y| \exp(-\lambda_{d-i+1}(x^*)t).
\label{eq:Psi_exponential_convergence}
\end{equation}

The following result describes the asymptotic directional behavior of the flow $\Psi$.

\begin{lemma}[Directional convergence  quadratic case]
\label{lem:directional_behavior_linear_flow}
Let $i=2,\hdots,d$,  $y=\sum_{j=d-i+1}^d \alpha_j v_j(x^*)\in {\mathcal S}_{i}$. Then, for every $t\in(0,\infty)$, we have
\begin{gather}
\label{eq:estimate_secant_Psi}
\left|\rho\left(\Psi_t(y)\right)-\text{{\rm sign}}(\alpha_{d-i+1})v_{d-i+1}(x^*)\right|\leq \frac{2|y|}{|\alpha_{d-i+1}|}\exp\left(-(\lambda_{d-i+2}(x^*)-\lambda_{d-i+1}(x^*))t\right), \\
\begin{split}
\label{eq:estimate_derivative_Psi}
& \; \left|\rho\left(\partial_t \Psi_t(y)\right)+\text{{\rm sign}}(\alpha_{d-i+1})v_{d-i+1}(x^*)\right| \\ \leq & \; \frac{2|y|}{|\alpha_{d-i+1}|}\left(\frac{\lambda_{d}(x^*)}{\lambda_{d-i+1}(x^*)}\right)\exp\left(-(\lambda_{d-i+2}(x^*)-\lambda_{d-i+1}(x^*))t\right).
\end{split}
\end{gather}
\end{lemma}

\begin{proof*}
We only prove~\eqref{eq:estimate_secant_Psi} as the proof of~\eqref{eq:estimate_derivative_Psi} is analogous. Since $\exp(\lambda_{d-i+1}(x^*)t) >0$, we have $\displaystyle \rho\left(\Psi_t(y)\right)=\rho\left(\frac{1}{|\alpha_{d-i+1}|}\exp(\lambda_{d-i+1}(x^*)t)\Psi_t(y)\right)$. It follows from~\eqref{eq:inegalite_basique_normalisee} and~\eqref{eq:Psi_t_explicit} that 
\begin{align*}
& \; \left|\rho\left(\Psi_t(y)\right)-\text{{\rm sign}}(\alpha_{d-i+1})v_{d-i+1}(x^*)\right|\\
= & \; \left|\rho\left(\frac{1}{|\alpha_{d-i+1}|}\exp(\lambda_{d-i+1}(x^*)t)\Psi_t(y)\right)-\rho(\text{{\rm sign}}(\alpha_{d-i+1})v_{d-i+1}(x^*))\right|\\
\leq & \; \frac{2}{|\alpha_{d-i+1}|}\left|\sum_{j=d-i+2}^d \alpha_j\exp(-(\lambda_j(x^*)-\lambda_{d-i+1}(x^*))t) v_{j}(x^*)\right| \\
\leq & \; \frac{2|y|}{|\alpha_{d-i+1}|}\exp\left(-(\lambda_{d-i+2}(x^*)-\lambda_{d-i+1}(x^*))t\right).
\end{align*}
Hence the result.\end{proof*} 

The next result shows, in particular, that, as $t\to\infty$, the directions of the trajectories $(\Phi_t(x))_{t\geq0}$ generically align, in terms of the initial condition $x\in B(x^*,\eta)$, with the tangent space to the talweg at $x^*$.

\begin{theorem}[Directional convergence of the gradient flow]
\label{HG_cas_continu}
Under Assumptions \ref{ass:H} and \ref{ass:min}, the following results  hold:\\
(i) {\rm ($C^1$ conjugacy)} For $\eta>0$ small enough there exists a local $C^1$-diffeomorphism $H\colon B(x^*,\eta)\to \RR^{d}$ satisfying $H(x^*)=0$ and $DH(x^*)=I_{d}$, such that
\be
H(\Phi_t(x))= \Psi_t(H(x)),\quad\text{for all }t\in [0,\infty) \text{ and } x\in B(x^*,\eta).
\label{conjugaison_cas_continu}
\ee
(ii) {\rm (Stable stratification)} Setting ${\mathcal S}_{i}^{H}(x^*):=H^{-1}({\mathcal S}_{i})$ for all $i=0,\hdots,d$,  ${\mathcal S}_{i}^{H}(x^*))$ is a 
manifold of dimension $i$,  $\{{\mathcal S}_{i}^{H}(x^*)\}_{i=0}^{d}$ forms a partition of $B(x^*,\eta)$ such that $\Phi_t({\mathcal S}_{i}^{H}(x^*))\subset {\mathcal S}_{i}^{H}(x^*)$ for all $t\geq0$\footnote{When no confusion can occur, we shall write ${\mathcal S}_{i}^{H}={\mathcal S}_{i}^{H}(x^*)$}.\\
(iii) {\rm (Directional convergence)} For all $t\in [0,\infty)$, and, for every $i=1,\hdots,d$ and $x\in {\mathcal S}_{i}^{H}$, the following hold:
\begin{align}
%& \exists\,c>0, \, \forall\,t\in[0,\infty)\quad\left|\Phi_t(x)-x^*\right|\leq c |H(x)| \exp(-\lambda_{d-i+1}(x^*)t),
%\label{eq:Phi_exp_conv}
& \left|\Phi_t(x)-x^*\right|= \mathrm O( \exp(-\lambda_{d-i+1}(x^*)t)),
\label{eq:Phi_exp_conv}\\
&\lim_{t \to \infty} \rho\left(\Phi_t(x)-x^*\right)=-\lim_{t \to \infty} \rho\left(\partial_t \Phi_t(x) \right)
 \in \{-v_{d-i+1}(x^*), v_{d-i+1}(x^*)\}.
\label{egalite_de_limites}
\end{align}
\end{theorem}

\begin{proof*} Let $\eta>0$ be as in Lemma~\ref{lem_Step_1}. It follows from~\cite[Theorem (I)]{MR141856} that, reducing $\eta$ if necessary, there exists $h\colon B(x^*,\eta)\to\RR^{d}$ of class $C^{1}$ such that $h(x^*)=0$, $Dh(x^*)=I_{d}$, and $h(\Phi_1(x))= \Psi_1(h(x))$ for all $x\in B(x^*,\eta)$. Set $A=-\nabla^{2}f(x^*)$ and define $H\colon B(x^*,\eta)\to\RR^{d}$ by
$$
H(x)=\int_{0}^{1}e^{-As}h\left(\Phi_s(x)\right)\dd s,\quad\text{for all }x\in B(x^*,\eta).
$$
From classical results $H$ is $C^1$. A
 change of variable as in the proof of~\cite[Lemma~4]{MR96853} (see also~\cite[Section 9.3.3]{MR4423370}), shows  that $H$ satisfies~\eqref{conjugaison_cas_continu}. Moreover, since $\Phi_s(x^*)=x^*$ standard computations show  that  $D_{x}\Phi_s(x^*)=e^{As}$ for all $s\in [0,\infty)$. On the other hand, 
$H(x^*)=0$ and, since $Dh(x^*)=I_{d}$, we have
$$
DH(x^*)=\int_{0}^{1}e^{-As}Dh(\Phi_s(x^*))D_{x}\Phi_s(x^*)\dd s=I_{d}.
$$
In particular, by the inverse function theorem, $H$ is a local $C^1$-diffeomorphism. This proves (i).

Let us prove (ii) and (iii). Note that~\eqref{eq:Psi_t_explicit} implies that, for every $i=1,\hdots,d$ and $t\in [0,\infty)$, $\Psi_t({\mathcal S}_{i})\subset {\mathcal S}_{i}$, which, by~\eqref{conjugaison_cas_continu}, implies that $\Phi_t({\mathcal S}_{i}^{H})\subset {\mathcal S}_{i}^{H}$ for all $t\in [0,\infty)$. Moreover, for every  $x\in B(x^*,\eta)$ and $t\in [0,\infty)$, we have $\Phi_t(x)-x^*=H^{-1}\left(\Psi_t(H(x))\right)-H^{-1}(0)$ and, hence,~\eqref{eq:Phi_exp_conv} follows from the boundedness of $DH^{-1}$ and~\eqref{eq:Psi_exponential_convergence}. 

Let $x\in {\mathcal S}_{i}^{H}$ with $i\geq1$. Note that since $i\neq0$, we have $\Phi_t(x)-x^*\neq 0$ and $\partial_t\Phi_t(x)\neq 0$ for all $t\geq0$ so that both directions can be evaluated. For every $j=1,\hdots,d$, set $\alpha_{j}=\langle H(x),v_{j}(x^*)\rangle$. For every $t\in (0,\infty)$ we have
\begin{equation}
\begin{split}
& \left|\rho\left(\Phi_t(x)-x^{*}\right)-\text{{\rm sign}}(\alpha_{d-i+1})v_{d-i+1}(x^*)\right|\\
\leq & \left|\rho\left(\Phi_t(x)-x^{*}\right)-\rho\left(\Psi_t(H(x))\right)\right|+\left|\rho\left(\Psi_t(H(x))\right)-\text{{\rm sign}}(\alpha_{d-i+1})v_{d-i+1}(x^*)\right|.
\label{eq:triangular_inequality_00}
\end{split}
\end{equation}
By~\eqref{eq:Psi_t_explicit} and the orthogonality of the family $\{v_{i}(x^*)\}_{i=1}^{d}$ we have 
$$
\left|\exp(\lambda_{d-i+1}(x^*)t)\Psi_t(H(x))\right|\geq |\alpha_{d-i+1}|>0.
$$
Thus,~\eqref{eq:inegalite_basique_normalisee} yields 
\begin{equation}
\begin{split}
& \left|\rho\left(\Phi_t(x)-x^{*}\right)-\rho\left(\Psi_t(H(x))\right)\right|\\
=& \left|\rho\left(\exp(\lambda_{d-i+1}(x^*)t)\left(\Phi_t(x)-x^{*}\right)\right)-\rho\left(\exp(\lambda_{d-i+1}(x^*)t)\Psi_t(H(x))\right)\right|\\
\leq & \frac{2\exp(\lambda_{d-i+1}(x^*)t)}{|\alpha_{d-i+1}|}\left|\Phi_t(x)-x^*-\Psi_t(H(x))\right|.
\label{eq:est_rho_phi_t_continu_1_bis}
\end{split}
\end{equation}
On the other hand, it follows from~\eqref{conjugaison_cas_continu} and $DH^{-1}(0)=I_{d}$ that 
\begin{equation}
\Phi_t(x)-x^*-\Psi_t(H(x))=\left(\int_{0}^{1}\left(DH^{-1}(\tau \Psi_t(H(x)))-DH^{-1}(0)\right)\dd \tau\right) \Psi_t(H(x)),
\label{eq:Phi_and_Psi}
\end{equation}
which, together with~\eqref{eq:est_rho_phi_t_continu_1_bis},~\eqref{eq:Psi_exponential_convergence}, and the continuity of $H^{-1}$, implies that 
\begin{equation}
\begin{split}
& \; \left|\rho\left(\Phi_t(x)-x^{*}\right)-\rho\left(\Psi_t(H(x))\right)\right|\\
 \leq & \; \frac{2|H(x)|}{|\alpha_{d-i+1}|}\int_{0}^{1}\left\|DH^{-1}(\tau \Psi_t(H(x)))-DH^{-1}(0)\right\|\dd \tau\underset{t\to\infty}{\longrightarrow} 0.
\label{eq:difference_Phi_Psi_continu_bis} 
\end{split}
\end{equation}
Therefore, we deduce from~\eqref{eq:triangular_inequality_00} and Lemma~\ref{lem:directional_behavior_linear_flow} that 
\begin{equation}
\left|\rho\left(\Phi_t(x)-x^* \right) - \text{{\rm sign}}(\alpha_{d-i+1}) v_{d-i+1}(x^*)\right|\underset{t\to\infty}{\longrightarrow} 0.
\label{final_bound_1}
\end{equation}
Similarly, we have 
\begin{equation}
\begin{split}
& \left|\rho\left(\partial_t \Phi_t(x) \right)+\text{{\rm sign}}(\alpha_{d-i+1})v_{d-i+1}(x^*)\right|\\[6pt]
\leq & \left|\rho\left(\partial_t \Phi_t(x) \right)-\rho\left(\partial_t \Psi_t(H(x)) \right)\right|
+\left|\rho\left(\partial_t \Psi_t(H(x)) \right)+\text{{\rm sign}}(\alpha_{d-i+1})v_{d-i+1}(x^*)\right|.
\label{eq:first_decomposition_rho_difference_Phi_continuous}
\end{split}
\end{equation}
By~\eqref{eq:Psi_t_explicit}, we have
\begin{gather*}
\forall\,t\in[0,\infty),\quad \partial_t \Psi_t(H(x))=-\sum_{j=d-i+1}^d\alpha_{j}\lambda_{j}(x^*)\exp(-\lambda_{j}(x^*)t)v_{j}(x^*).
\end{gather*}
The orthogonality of the family $\{v_{i}(x^*)\}_{i=1}^{d}$ implies that
\begin{gather*}
|\lambda_{d-i+1}(x^*)^{-1}\exp(\lambda_{d-i+1}(x^*)t)\Psi_t(H(x))|\geq |\alpha_{d-i+1}|>0.
\end{gather*}
Thus,~\eqref{eq:inegalite_basique_normalisee} yields 
\begin{equation}
\left|\rho\left(\partial_t \Phi_t(x) \right)-\rho\left(\partial_t \Psi_t(H(x)) \right)\right|\leq\frac{2\exp(\lambda_{d-i+1}(x^*)t)}{\lambda_{d-i+1}(x^*)|\alpha_{d-i+1}|}\left|\partial_t \Phi_t(x)-\partial_t \Psi_t(H(x))\right|.
\label{eq:diff_der_Phi_continuous}
\end{equation}
Since~\eqref{conjugaison_cas_continu} implies $
DH(\Phi_t(x))\partial_t \Phi_t(x)= \partial_t \Psi_t(H(x))$, we deduce that 
\begin{equation}
\left|\partial_t \Phi_t(x)-\partial_t \Psi_t(H(x))\right|
\leq\left|\partial_t \Psi_t(H(x))\right| \|DH^{-1}(\Psi_t(H(x)))-DH^{-1}(0)\|,
\label{eq:diff_der_Phi_continuous_bis}
\end{equation}
which, together with~\eqref{eq:first_decomposition_rho_difference_Phi_continuous},~\eqref{eq:diff_der_Phi_continuous},~\eqref{eq:Psi_exponential_convergence}, the continuity of $H^{-1}$, and Lemma~\ref{lem:directional_behavior_linear_flow}, implies that 
\begin{equation}
\left|\rho\left(\partial_t \Phi_t(x) \right) + \text{{\rm sign}}(\alpha_{d-i+1}) v_{d-i+1}(x^*)\right|\underset{t\to\infty}{\longrightarrow} 0.
\label{final_bound_2}
\end{equation}
Therefore,~\eqref{egalite_de_limites} follows from~\eqref{final_bound_1} and~\eqref{final_bound_2}.
\end{proof*}

\begin{theorem}\label{th:align}
If \Cref{ass:H} holds at each critical point then: 

(i) {\rm (Generic alignment)} If $f$ is coercive, the set of initial conditions whose gradient flow trajectories converge to local minimizers of $f$ with the velocities and the secant aligning to the talweg has full Lebesgue measure.

(ii) {\rm (Alignment rate)}  Assume that $f$ is $C^\infty${\rm(\footnote{Actually we may assume that $f$ is of class $C^k$ with $k$ sufficiently large, see \cite{MR96853} for more details.})}, let  $x$ be an initial condition so that, $\Phi_t(x)\to x^*$ as $t\to\infty$ with $\Phi_t(x)\in {\mathcal S}_d^H(x^*)$ for some $t\geq 0${\rm(\footnote{This corresponds to the generic case in (i).})}, and suppose, in addition, the non resonance conditions at $x^*$: for all integers $m_{1},\hdots,m_{d}$ such that $\sum_{j=2}^{d}m_{j}\geq 2$, we have\footnote{Observe this is a generic condition as the $m_i$ are integers.}
$$
\forall\; i=1,\hdots,d,\quad\lambda_{i}(x^*)\neq\sum_{j=1}^{d}m_{j}\lambda_{j}(x^*).
$$
Then the trajectory  aligns  with the tangent to the talweg at a rate governed by:
\begin{align*}
\mathrm{dist}\!\Big(
\rho(\partial_t \Phi_t(x)),
\,T_{x^*}\tal([r^*,\bar r))
\Big)
&=\mathrm{O}\!\left(\exp(-(\lambda_2-\lambda_1)t)) + \mathrm{O}(\exp(-\lambda_1 t)\right),\\
\mathrm{dist}\!\Big(
\rho(\Phi_t(x)-x^*),
\,T_{x^*}\tal([r^*,\bar r))
\Big)
&=\mathrm{O}\!\left(\exp(-(\lambda_2-\lambda_1)t)) + \mathrm{O}(\exp(-\lambda_1 t)\right),
\end{align*}
where $\bar r>r^*$ is sufficiently close to $r^*$.
\end{theorem}

\begin{proof}
Item {\rm(i)} extends \cite{thom1949partition} and also \cite{pemantle1990nonconvergence,lee2016gradient}. 

 In view of \Cref{ass:H} all the critical points of $f$ are isolated; hence the set of critical points is at most countable. 
Local uniqueness of critical points entail the well known fact:
\begin{claim} By coercivity, $f$ is inf-bounded so the flow $\Phi_t$ is defined for all $t\geq 0$, and moreover one can prove that for all $x$ in $\RR^d$, $\Phi_t(x)$ converges to some $x^*$ belonging to the critical set $\crit f=(\nabla f)^{-1}(\{0\})$. %\jer{trouver ref}
\end{claim}
 
 Let us partition $\crit f$ into the set of local minimizers $\argminloc f$ and the set of  unstable critical points $\crit^u f$, i.e., those whose Hessian has a negative eigenvalue. Whence $\crit f=\argminloc f\cup \crit^u f.$\\
Let us recall two facts:
\begin{itemize}
\item[(a)] For each \(x^*\in \argminloc f\), with the notation \(\eta:=\eta_{x^*}\) of \Cref{HG_cas_continu}, we have
\[
B(x^*,\eta)=\bigcup_{i=1}^d {\mathcal S}_i^H(x^*).
\]
\item[(b)] For each \({x^*}\in \crit^u f\), using \Cref{ass:H}, we may apply the Hadamard--Perron decomposition into stable and unstable manifolds, see \cite{MR4423370}. Let \(W^{s}_{x^*}\) denote the local stable manifold of \({x^*}\), which has dimension strictly less than \(d\) as $x^*$ is not a local minimizer, and which is defined locally as 
$$W^s_{x^*}=\{x\in B(x^*,\varepsilon_{x^*}): \Phi_t(x)\to x^* \mbox{ as $t\to \infty$}\} \mbox{ for $\varepsilon_{x^*}>0$ small enough.}$$
\end{itemize}
We then define the following two subsets of $\RR^d$
\begin{align}
\NNN&=
\bigcup_{x^*\in \argminloc f}\bigcup_{i=1}^{d-1}\bigcup_{t\in\mathbb N}\Phi_t^{-1}\bigl({\mathcal S}_i^H(x^*)\bigr)
\;\cup\;
\bigcup_{x^*\in \crit^u f}\bigcup_{t\in\mathbb N}\Phi_t^{-1}\bigl(W^{s}_{x^*}\bigr),\label{eq:defNproof}\\
\FFF &=
\bigcup_{x^*\in \argminloc f}\bigcup_{t\in\mathbb N}\Phi_t^{-1}\bigl({\mathcal S}_d^H(x^*)\bigr).\label{eq:defFproof}
\end{align}

Let us observe (and check) that we have the partition
\begin{equation}\label{eq:part}
\mathbb R^d=\NNN\cup \FFF.
\end{equation}

By the fact that $\{{\mathcal S}_{i}^{H}\}_{i=1,\ldots,d}$ is a partition, we have $\NNN\cap\FFF=\emptyset$. Let  \(x\) in \(\mathbb R^d\). We know by Claim 1 that \(t\mapsto \Phi_t(x)\) converges to some critical point \(x^*\in \crit f\). Two cases may occur:

\noindent{Case 1: \(x^*\in \crit^u f\).} 
Since \(\Phi_t(x)\to x^*\), by definition of the stable manifold of \(x^*\), there exists \(T>0\) such that
$\Phi_t(x)$ in $W^{s}_{x^*}$ for all  $t\ge T.$
Thus, for any integer \(q\ge T\), \(\Phi_q(x)\in W^{s}_{x^*}\), and therefore
\[
x\in \Phi_q^{-1}(W^{s}_{x^*})\subset \NNN.
\]

\medskip

\noindent{Case 2: \(x^*\in \argminloc f\).} 
There exists \(T>0\) such that, for any  \(t\ge T\),
\[
\Phi_t(x)\in B(x^*,\eta_{x^*})=\bigcup_{i=1}^d {\mathcal S}_i^H(x^*).
\]
Take an integer $q\geq T$. Since \(\{{\mathcal S}_i^H(x^*)\}_{i=1}^d\) forms a partition, we have either \(\Phi_q(x)\in {\mathcal S}_d^H(x^*)\) or \(\Phi_q(x)\in \bigcup_{i=1}^{d-1} {\mathcal S}_i^H(x^*)\). Hence either \(x\in \FFF\) or \(x\in \NNN\). This proves \eqref{eq:part}.

\medskip

Let us observe that \(\NNN\) has Lebesgue measure zero. This follows from the fact that each \(\Phi_t\) is a diffeomorphism, and that the sets \({\mathcal S}_i^H(x^*)\) for \(i=1,\ldots,d-1\) and \(W^{s}_{x^*}\) are manifolds of dimension strictly less than \(d\). Hence \(\NNN\) is a countable union of sets of Lebesgue measure zero, and therefore has Lebesgue measure zero.

To conclude we just need to observe that $x\in \FFF$ implies that $x$ converges to some local minimizer $x^*$ with in addition $\Phi_t(x)\in {\mathcal S}_d^H$, so that \Cref{HG_cas_continu} applies and we have directional alignment of the secant and the velocity with the talweg.
 
Let us prove {\rm(ii)}. Recall that since $x^*$ is a local minimizer the tangent space to the talweg is $\RR v_1(x^*)$. Shifting time if necessary, we may assume   
$x\in {\mathcal S}_d^H$. Thus with the notation of \Cref{HG_cas_continu}, we have $\alpha_1\neq 0$ in the expression $H(x)=\sum_1^d  \alpha_i v_i(x^*)$.

By Sternberg's theorem \cite[Theorem 2]{MR96853}, under the non-resonance assumptions the conjugating diffeomorphism $H$ may be assumed to be of class $C^2$. In particular, shrinking neighborhoods if necessary, there exists $C>0$ such that $DH$ and $DH^{-1}$ are locally $C$-Lipschitz continuous.

We estimate the alignment rate for the direction of the secant. From \eqref{eq:triangular_inequality_00}, we have
\begin{equation}\label{eq:proof_align_secant_1}
\begin{split}
&\left|\rho\bigl(\Phi_t(x)-x^*\bigr)-\mathrm{sign}(\alpha_1)v_1(x^*)\right|\\
\leq\;&
\left|\rho\bigl(\Phi_t(x)-x^*\bigr)-\rho\bigl(\Psi_t(H(x))\bigr)\right|
+
\left|\rho\bigl(\Psi_t(H(x))\bigr)-\mathrm{sign}(\alpha_1)v_1(x^*)\right|.
\end{split}
\end{equation}
Using \eqref{eq:difference_Phi_Psi_continu_bis} and \eqref{eq:estimate_secant_Psi}, we obtain
\begin{align}
&\left|\rho\bigl(\Phi_t(x)-x^*\bigr)-\mathrm{sign}(\alpha_1)v_1(x^*)\right| \nonumber\\
&\leq
\frac{2|H(x)|}{|\alpha_1|}
\int_0^1
\left\|DH^{-1}\bigl(\tau\Psi_t(H(x))\bigr)-DH^{-1}(0)\right\|\,\dd\tau
+
\frac{2|H(x)|}{|\alpha_1|}\exp(-(\lambda_2-\lambda_1)t).
\label{eq:proof_align_secant_2}
\end{align}
Since $DH^{-1}$ is locally Lipschitz,
\[
\int_0^1
\left\|DH^{-1}\bigl(\tau\Psi_t(H(x))\bigr)-DH^{-1}(0)\right\|\,\dd\tau
\leq C |\Psi_t(H(x))|.
\]
Thus, using $|\Psi_t(H(x))|=\mathrm{O}(\exp(-\lambda_1t))$, yields
\begin{equation}\label{eq:proof_align_secant_3}
\left|\rho\bigl(\Phi_t(x)-x^*\bigr)-\mathrm{sign}(\alpha_1)v_1(x^*)\right|
\leq
\mathrm{O}(\exp(-\lambda_1 t))
+
\frac{2|H(x)|}{|\alpha_1|}\exp(-(\lambda_2-\lambda_1)t).
\end{equation}
Since $\mathrm{sign}(\alpha_1)v_1(x^*)\in T_{x^*}\tal([r^*,\bar r))$, we infer
\begin{equation}\label{eq:proof_align_secant_4}
\mathrm{dist}\!\left(
\rho(\Phi_t(x)-x^*),\,T_{x^*}\tal([r^*,\bar r))
\right)
=
\mathrm{O}\!\left(\exp(-(\lambda_2-\lambda_1)t)\right)
+
\mathrm{O}\!\left(\exp(-\lambda_1 t)\right).
\end{equation}

The proof for the velocity alignment is similar as it also relies on the local Lipschitz continuity of $DH^{-1}$, on the inequalities in the proof of \Cref{HG_cas_continu} and on the estimate \eqref{eq:estimate_derivative_Psi}.\end{proof}

\begin{remark}[On the sharpness of the alignment rate]\label{rem:alignement-rate-cont}
In the quadratic case the directional convergence is governed by the spectral gap $\lambda_2-\lambda_1$. One might therefore wonder whether the alignment rate is always $\exp(-(\lambda_2-\lambda_1)t)$. In the nonlinear world the slower coordinate, corresponding to the small eigenvalue, can act back on the fastest one through the nonlinearity and  eventually mitigate the influence of $\lambda_2$. This is what happens with $f\colon\RR^2\to\RR$, defined by
\[
f(x_1,x_2)=\frac{1}{2}\lambda_1x_1^2+\frac{1}{2}\lambda_2x_2^2+a x_1^2x_2,
\quad a\neq0,
\mbox{ where }
0<2\lambda_1<\lambda_2.
\]
The origin is a nondegenerate critical point and
$
\nabla^2 f(0)=
\begin{pmatrix}
\lambda_1 & 0\\
0 & \lambda_2
\end{pmatrix}$. Thus, the tangent space to the talweg is $T_0\extun=\RR(1,0)$. Let $x_0=(\alpha_{1},\alpha_{2})\in\RR^{d}$, with $\alpha_1\neq 0$, and for every $t\in [0,\infty)$, set $x(t)=\Phi_{t}(x_0)$. We prove in Section~\ref{subsec:some_proof} in the Appendix that, provided that $|x_0|$ is small enough, 
\begin{equation}
\dist\!\left(
\rho(x(t)),
\,T_0\extun
\right)
= C\exp(-\lambda_1 t)+\mathrm o(\exp(-\lambda_1 t)), \text{ for some } C>0.
\label{eq:estimation_rho_x_t}
\end{equation}
Since here $\lambda_2-\lambda_1>\lambda_1$, the spectral gap would predict a faster decay of order $\exp(-(\lambda_2-\lambda_1)t)$. This example shows that such a rate cannot be expected in general.
\end{remark}

\begin{corollary}[Valley is generically reached in finite time]
\label{on_va_vers_V_eps_bis}	
Under Assumptions \ref{ass:H} and \ref{ass:min}, let $\eta>0$ and $H\colon B(x^*,\eta)\to\RR^{d}$ be as in Theorem~\ref{HG_cas_continu}. Then, for every $w>0$ and $x\in {\mathcal S}_{d}^{H}$, there exists $t_0\in [0,\infty)$ such that 
$$
\Phi_t(x)\in \vall_{f,w}(x^*),\quad\text{for all }t\geq t_{0}.
$$ 
\end{corollary}
\begin{proof*}
Let $0<w_{1}<w$. By Proposition~\ref{valley_inclusions}\,{\rm(i)}, there exists $\delta\in (0,\eta)$ such that
\begin{equation}
x^*+\bigl(\vall_{q,w_1}\cap B(0,\delta)\bigr)\subset \vall_{f,w}(x^*).
\label{eq:inclusions_continu}
\end{equation}
Assume for instance that $\displaystyle \rho(\Phi_t(x)-x^*)\to v_1$ as $t\to+\infty$. Since $\vall_{q,w_1}$ is a cone generated by $v_1$, this implies that for $t$ large enough, $\displaystyle \Phi_t(x)-x^*\in \vall_{q,w_1}$. On the other hand, we know that $\displaystyle \Phi_t(x)\to x^*$, as $t\to +\infty$, therefore, for $t$ sufficiently large, $\displaystyle \Phi_t(x)-x^*\in \vall_{q,w_1}\cap B(0,\delta)$. Using \eqref{eq:inclusions_continu}, we conclude that $\displaystyle \Phi_t(x)\in \vall_{f,w}(x^*)$ for all $t$ large enough.
\end{proof*}
 
Given a Lebesgue measurable set $S$, we denote by $\vol(S)=\displaystyle\int_S1\dd x$ its volume.
\begin{theorem}[Volume concentration in  valleys]
\label{rapport_de_volumes}	
Suppose that Assumptions  \ref{ass:H} and \ref{ass:min} hold and let $\eta>0$ be small enough. Then for every  measurable set $S \subset  B(x^*,\eta)$ of initial conditions, with $\vol(S)>0$, and $w>0$, we have

\begin{equation}
\lim_{t\to \infty}\frac{\vol\left(\Phi_t(S) \cap \vall_{f,w}(x^*)\right)}{\vol\left(\Phi_t(S)\right)}=1 .  
\label{eq:rapport_de_volumes}	
\end{equation}
\end{theorem}
\begin{figure}[ht]
\includegraphics[width=\linewidth]{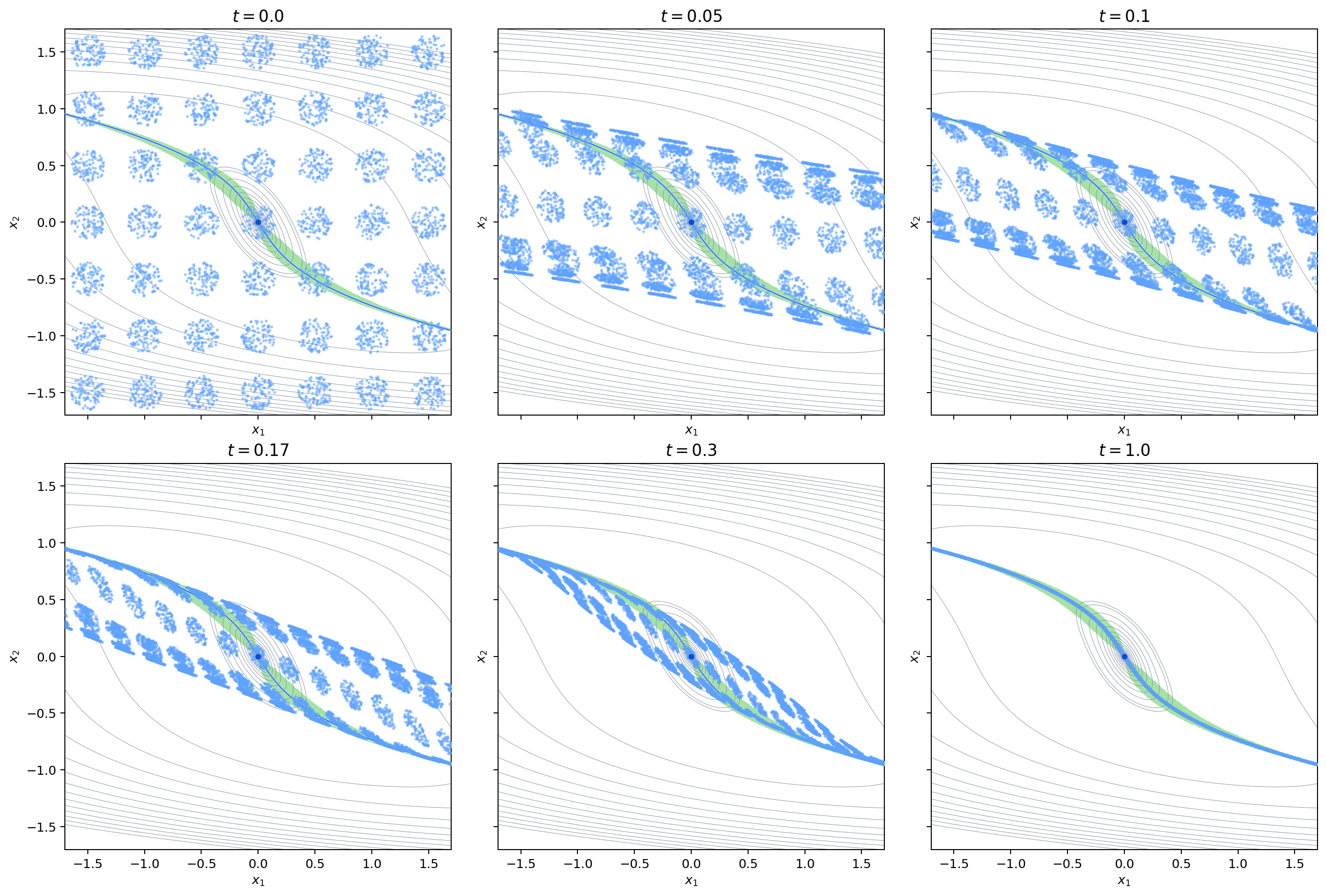}
\caption{Uniformly distributed balls form the set  $S$ on $[-5,5]^2$ of \Cref{rapport_de_volumes}. They are randomly sampled and deformed by the gradient flow.  This illustrates the concentration phenomenon of formula \eqref{eq:rapport_de_volumes} and provides some mathematical insights into the ``formation of rivers".}
\end{figure}

\begin{proof*} 
For every $t\in[0,\infty)$, $\Phi_t$ is a local diffeomorphism and, for every $x\in\RR^{d}$, setting $Z(t,x)=\det  \partial_x \Phi_t(x)$,  Jacobi's formula implies   
$$\ba{rcl}
\ds \partial_t Z(t,x)&=& - \tr\left( \nabla^2f \left(\Phi_t(x)\right) \right)Z(t,x)
=- \ds \left(\sum_{i=1}^{d} \lambda_i \left(\Phi_t(x)\right)\right)Z(t,x),\\[10pt]
Z(0,x)&=& 1.
\ea
$$
In turn, 
$$
\det  \partial_x \Phi_t(x)= \exp\left(\ds -\int_{0}^{t} \sum_{i=1}^{d} \lambda_i \left(\Phi_s(x)\right) \dd s \right),\quad\text{for all }(t,x)\in [0,\infty)\times\RR^{d}.
$$
On the other hand, it follows from Theorem~\ref{HG_cas_continu}, Lemma~\ref{lem_Step_1}, and Corollary~\ref{cordiff} in the Appendix, one can shrink $\eta$ to obtain a local diffeomorphism $H\colon B(x^*,\eta)\to\RR^{d}$ such that $H(x^*)=0$, $DH(x^*)=I_{d}$, $\Phi_t(B(x^*,\eta )) \subset B(x^*,\eta ),$ for all $t\in [0,\infty)$, the conjugacy relation~\eqref{conjugaison_cas_continu} holds, and
\begin{equation}
\sup_{x\in B(x^*,\eta )}|H(x)|+\max_{i=1,\hdots,d}\sup_{x\in B(x^*,\eta )}|\nabla\lambda_{i}(x)|+\sup_{y\in N}\|D H^{-1}(y)\|<\infty,  
\label{eq:preliminary_bounds}
\end{equation}
where $N$ is a closed convex neighborhood containing  $H(B(x^*,\eta ))$ and on which $DH^{-1}$ is well defined. In particular, by the mean value theorem there exists $C_1>0$ such that, for every $x\in B(x^*,\eta )$ and $t\in [0,\infty)$,
\begin{multline*}
\exp\left(-t\sum_{i=1}^{d}\lambda_{i}(x^*)-C_{1}\int_{0}^{t}|\Phi_s(x)-x^*|\dd s\right)\leq\det  \partial_x \Phi_t(x)\\
\leq \exp\left(-t\sum_{i=1}^{d}\lambda_{i}(x^*)+ C_{1}\int_{0}^{t}|\Phi_s(x)-x^*|\dd s\right).
\end{multline*} 
Using~\eqref{eq:Phi_exp_conv} and~\eqref{eq:preliminary_bounds}, we get $\sup_{x\in B(x^*,\eta )} \int_{0}^{\infty}|\Phi_s(x)-x^*| \dd s<\infty$
and, hence, there exists $\zeta>0$ such that, for every $x\in B(x^*,\eta )$ and $t\in(0,\infty)$, 
\begin{equation}
\label{estimate_jacobian_continuous_case}
\exp\left(-t\sum_{i=1}^{d}\lambda_{i}(x^*)-\zeta\right)\leq\det  \partial_x \Phi_t(x)\leq \exp\left(-t\sum_{i=1}^{d}\lambda_{i}(x^*)+\zeta\right).
\end{equation}

Finally, let $S\subset B(x^*,\eta )$ be a measurable set, with $\vol(S)>0$, and $w>0$. Since, for every $t\in[0,\infty)$, $\Phi_t$ is a local diffeomorphism, we have $\vol\left(\Phi_t(S)\right)>0$ and 
\begin{equation}
 \frac{\vol\left(\Phi_t(S) \cap   \vall_{f,w}(x^*)\right)}{\vol\left(\Phi_t(S)\right)}= 1- R(t),\quad\text{for all }t\in [0,\infty),
\label{eq:rapport_de_volumes_with_error}
\end{equation}
 where, by the change of variable formula,
$$
R(t):=\frac{\vol\left(\Phi_t(S) \setminus\vall_{f,w}(x^*)\right)}{\vol\left(\Phi_t(S)\right)}
% =\frac{\int_{\Phi_t(S) \setminus\vall_{f,w}(x^*)}\dd x}{\int_{\Phi_t(S)}\dd x}
=\frac{\displaystyle\int_{S\setminus\Phi_t^{-1}( \vall_{f,w}(x^*))}\det\left( \partial_x \Phi_t(x)\right)\dd x}{\int_{S}\det\left( \partial_x \Phi_t(x)\right)\dd x}.
$$
Estimate~\eqref{estimate_jacobian_continuous_case} yields
$$
R(t)\leq \exp(2\zeta)\frac{\vol\left(S \setminus\Phi_t^{-1}(\vall_{f,w}(x^*))\right)}{\vol(S)},\quad\text{for all }t\in(0,\infty).
$$
In turn, it follows from Corollary~\ref{on_va_vers_V_eps_bis}, $\vol(B(x^*,\eta )\setminus {\mathcal S}_{d}^{H})=0$, and dominated convergence that $R(t)\to 0$, as $t\to\infty$, and, hence,~\eqref{eq:rapport_de_volumes} follows from~\eqref{eq:rapport_de_volumes_with_error}.
\end{proof*}

%%%%%%%%%%%%%%%%%%%%%%%%%%%%%%%%%%%%%%%%%%%%%%%%%%%%%%%
\subsection{Directional convergence for saddle points} 
The previous section was centered on the behavior of the flow for the convergence toward local minimizers; the same analysis may be applied, in a Riemannian fashion, to any generic saddle point \(x^*\), i.e., for which \(\lambda_1(x^*)\lambda_d(x^*)<0\), after restricting \(f\) to its local stable manifold \(W^s_{x^*}\); see \cite{MR4423370} for theoretical elements. 

Indeed, if we endow \(W^s_{x^*}\) with the induced Euclidean metric of \(\RR^d\), then \(W^s_{x^*}\) is a smooth Riemannian manifold and the restricted function \(\hat f:=f_{|W^s_{x^*}}\) has \(x^*\) as a strong local minimizer. Moreover, the Riemannian gradient flow of \(\hat f\) is exactly the restriction of the ambient gradient flow to \(W^s_{x^*}\). Hence the notions introduced in Section~\ref{preliminaries}, gradient extremals, talweg, and valleys, are Riemannian objects that are well defined on $W^s_{x^*}, \langle\cdot,\cdot\rangle$, and the alignment and concentration results of this section may be developed in this non-Euclidean setting.

\medskip

Recent results also show that maxima, minima, and alignment phenomena share interesting connections. For this, we also refer the reader to~\cite{Belabbas_23}, where, using Hartman’s theorem as in the proof of Theorem~\ref{HG_cas_continu}, it is shown that if \(x^*\) is a nondegenerate local maximum, then, generically, one can associate to \(x^*\) two local minimizers, \(m_{-}(x^*)\) and \(m_{+}(x^*)\), such that gradient trajectories starting near \(x^*\) converge, with high probability, to one of these two minima.

%%%%%%%%%%%%%%%%%%%%%%%%%%%%%%%%%%%%%%%%%%%%%%%%%%%%%%%
\section{Directional convergence for the gradient method}
\label{sec:discret}

As convergence of the gradient method generically occurs to a local minimizer \cite{lee2016gradient},  we work throughout this section under \Cref{ass:min} with $x^*$ being a local minimizer. We shall systematically use the constants  $L_{1}, L_2>0$ defined through 
\begin{equation}
L_{2}:=\sup_{x\in B(0,\eta)}\|\nabla^{2}f(x)\|<\infty,
\label{eq:bounded_hessian}
\end{equation}
and 
\be
\langle \nabla^{2}f(x)h,h\rangle\geq L_{1}|h|^{2},
\label{eq:coercivity}
\ee
for any $x\in B(x^*,\eta)$ and $h\in\RR^{d}.$

Given $\gamma>0$, we consider the constant step gradient method
\begin{equation}
x\in B(x^*,\eta),\;\;x_0:=x,\; x_{k+1}=x_{k}-\gamma\nabla f(x_{k}),
\label{def:steepest_descent_iterates}
\end{equation}

for any $k\geq0.$ Given $k\in\NN$ and $x\in B(x^*,\eta),$ we denote by $\Phi_{k}(x)$ the $k^{\mathrm th}$ iterate in \eqref{def:steepest_descent_iterates}.
 
We will need the following standard result on the discrete flow $\Phi_k$. For the sake of completeness, we provide its short proof.

\begin{lemma}[Gradient descent stability for strong minimizers]
Suppose that $\gamma\in (0,2/L_{2})$ and define $s(\gamma)\in (0,1)$ by
$$
s(\gamma)=\max\{|1-\gamma L_{1}|,|1-\gamma L_{2}|\}=\begin{cases}  1-L_1\gamma, &\text{if }\gamma\in \left(0,\frac{2}{L_1+L_2}\right],\\[6pt]
                           L_2\gamma-1, &\text{if }\gamma\in \left(\frac{2}{L_1+L_2},\frac{2}{L_2}\right).
\end{cases}
$$
Then $\Phi_1$ is $s(\gamma)$-Lipschitz on $B(x^*,\eta)$ and
\begin{equation}
(\forall\,k\geq 0)\quad|\Phi_{k+1}(x)-x^*|\leq s(\gamma)|\Phi_{k}(x)-x^*|.
\label{eq:stability_discrete_flow}
\end{equation}
In particular, $\Phi_{k}(B(x^*,\eta)) \subset B(x^*,\eta)$ for all $k\in\NN$.
\label{lem:stability_discrete_flow}
\end{lemma}
\begin{proof*}
It follows from ~\eqref{eq:bounded_hessian} and~\eqref{eq:coercivity} that, for every $x\in B(x^*,\eta)$, the spectrum of the symmetric matrix $D\Phi_{1}(x)=I_{d}-\gamma\nabla^{2}f(x)$ is contained in $[-s(\gamma),s(\gamma)]$, which implies that $\Phi_1$ is $s(\gamma)$-Lipschitz on $B(x^*,\eta)$. In particular, for every $x\in B(x^*,\eta)$,
$$
|\Phi_{k+1}(x)-x^*|=|\Phi_{1}(\Phi_{k}(x))-\Phi_{1}(x^*)|\leq s(\gamma)|\Phi_{k}(x)-x^*|,
$$
from which~\eqref{eq:stability_discrete_flow} follows. 
\end{proof*}

\medskip

For every $k\in\NN$ and $y\in\RR^{d}$, denote by $\Psi_{k}(y)$ the $k^{\mathrm th}$ iterate of the sequence
\begin{equation}
y_{0}=y,\quad y_{k+1}=y_{k}-\gamma\nabla^{2}f(x^*) y_{k},
\label{def:steepest_descent_iterates_quadratic_case}
\end{equation}
i.e. $y_k=(I_{d}-\gamma\nabla^2f(x^*))^{k}y_0$ for all $k\in\NN$. Recall that the manifolds ${\mathcal S}_{i}$ ($i=0,\hdots,d$) are defined in~\eqref{lem:partition}. Since $\nabla^{2}f(x^*)=P(x^*)\mathscr{D}(x^*)P^{\top}(x^*)$, for every $i=1,\hdots,d$ and $y=\sum_{j=d-i+1}^d \alpha_j v_j(x^*)\in {\mathcal S}_{i}$, we have for all $k\geq 0$, 
\begin{equation}
\Psi_{k}(y)=\sum_{j=d-i+1}^{d}(1-\gamma\lambda_{j}(x^*))^{k}\alpha_{j}v_{j}(x^*)
\label{eq:quadratic_steepest_descent_iterates}
\end{equation}
and hence, if $\gamma\in(0,2/\lambda_{d}(x^*))$, one  has 
\begin{equation}
\left|\Psi_{k}(y)\right|\leq \kappa_{i}(\gamma)^{k}|y|,\quad\text{for all }k\geq 0,
\label{eq:borne_Psi_k_cas_discret}
\end{equation}
where $\kappa_{i}(\gamma)\in(0,1)$ is given by
\begin{equation}
\kappa_{i}(\gamma)=\begin{cases}
1-\gamma\lambda_{d-i+1}(x^*),&\text{if }\gamma\in\left(0,\frac{2}{\lambda_{d-i+1}(x^*)+\lambda_{d}(x^*)}\right],\\
\gamma\lambda_{d}(x^*)-1,&\text{if }\gamma\in\left(\frac{2}{\lambda_{d-i+1}(x^*)+\lambda_{d}(x^*)},\frac{2}{\lambda_{d}(x^*)}\right).
\end{cases}
\label{eq:def_kappa_gamma}
\end{equation}

The following result provides the asymptotic directional behavior of the discrete flow $\Psi$.

\begin{lemma}
Let $i=2,\hdots,d$, let $y=\sum_{j=d-i+1}^d \alpha_j v_j(x^*)\in {\mathcal S}_{i}$, and take $\gamma\in \left(0,\frac{2}{\lambda_{d-i+1}(x^*)+\lambda_{d}(x^*)}\right)$. Then, for every $k\geq 0$, we have
\begin{align}
\left|\rho\left(\Psi_{k}(y)\right)-\text{{\rm sign}}(\alpha_{d-i+1})v_{d-i+1}(x^*)\right|&\leq \frac{2|y|}{|\alpha_{d-i+1}|}\kappa_{{\rm dir},i}(\gamma)^{k},
\label{eq:estimate_secant_Psi_discrete}\\
\left|\rho\left(\frac{\Psi_{k+1}(y)-\Psi_{k}(y)}{\gamma}\right)+\text{{\rm sign}}(\alpha_{d-i+1})v_{d-i+1}(x^*)\right|&\leq  \frac{2|y|}{|\alpha_{d-i+1}|}\left(\frac{\lambda_{d}(x^*)}{\lambda_{d-i+1}(x^*)}\right)\kappa_{{\rm dir},i}(\gamma)^{k},
\label{eq:estimate_derivative_Psi_discrete}
\end{align}
where $\kappa_{{\rm dir},i}(\gamma)\in(0,1)$ is given by 
\begin{equation}
\kappa_{{\rm dir},i}(\gamma)=\begin{cases}
\frac{1-\gamma\lambda_{d-i+2}(x^*)}{1-\gamma\lambda_{d-i+1}(x^*)}, &\text{if }\gamma\in \left(0,\frac{2}{\lambda_{d-i+2}(x^*)+\lambda_{d}(x^*)}\right],\\[10pt]
\frac{\gamma\lambda_{d}(x^*)-1}{1-\gamma\lambda_{d-i+1}(x^*)}, &\text{if }\gamma\in \left(\frac{2}{\lambda_{d-i+2}(x^*)+\lambda_{d}(x^*)},\frac{2}{\lambda_{d-i+1}(x^*)+\lambda_{d}(x^*)}\right).
\end{cases}
\label{def:kappa_dir}
\end{equation}
\label{lem:directional_behavior_linear_flow_discrete}
\end{lemma}
\begin{proof*}
We only show~\eqref{eq:estimate_derivative_Psi_discrete} as~\eqref{eq:estimate_secant_Psi_discrete} follows from similar computations. By ~\eqref{eq:quadratic_steepest_descent_iterates}, we have  
\begin{multline}
\frac{\Psi_{k+1}(y)-\Psi_{k}(y)}{\gamma}=-\alpha_{d-i+1}\lambda_{d-i+1}(x^*)(1-\gamma\lambda_{d-i+1}(x^*))^{k}v_{d-i+1}(x^*)\\
-\sum_{j=d-i+2}^{d}\alpha_{j}\lambda_{j}(x^*)(1-\gamma\lambda_{j}(x^*))^{k}v_{j}(x^*).
\label{eq:calcul_difference_psi}
\end{multline}
On the other hand, since $\lambda_{d-i+1}(x^*)>0$ and $1-\gamma\lambda_{d-i+1}(x^*)>0$, we have 
$$
\rho\left(\frac{\Psi_{k+1}(y)-\Psi_{k}(y)}{\gamma}\right)=\rho\left(\lambda_{d-i+1}(x^*)^{-1}(1-\gamma\lambda_{d-i+1}(x^*))^{-k}\frac{\Psi_{k+1}(y)-\Psi_{k}(y)}{\gamma}\right)
$$
and, hence,~\eqref{eq:inegalite_basique_normalisee} yields
\begin{equation}
\begin{split}
& \; \left|\rho\left(\frac{\Psi_{k+1}(y)-\Psi_{k}(y)}{\gamma}\right)+\text{{\rm sign}}(\alpha_{d-i+1})v_{d-i+1}(x^*)\right|\\
\leq & \; \left|\rho\left(\frac{\Psi_{k+1}(y)-\Psi_{k}(y)}{\gamma}\right)-\rho(-\alpha_{d-i+1}v_{d-i+1}(x^*))\right|\\
\leq & \; \frac{2}{|\alpha_{d-i+1}|}\left|\lambda_{d-i+1}(x^*)^{-1}(1-\gamma\lambda_{d-i+1}(x^*))^{-k}\frac{\Psi_{k+1}(y)-\Psi_{k}(y)}{\gamma}+\alpha_{d-i+1}v_{d-i+1}(x^*)\right|.
\label{eq:est_rho_difference_phi_k_discrete_2}
\end{split}
\end{equation}
Combining with~\eqref{eq:calcul_difference_psi} and noting that 
$$
\kappa_{{\rm dir},i}(\gamma)=\max\left\{\left|\frac{1-\gamma\lambda_{j}(x^*)}{1-\gamma\lambda_{d-i+1}(x^*)}\right|\;:\;j=d-i+2,\hdots,d\right\},$$
we get 
\small
\begin{equation}
\begin{split}
&\;\left|\rho\left(\frac{\Psi_{k+1}(y)-\Psi_{k}(y)}{\gamma}\right)+\text{{\rm sign}}(\alpha_{d-i+1})v_{d-i+1}(x^*)\right|\\
 \leq &\frac{2}{|\alpha_{d-i+1}|}\left|\sum_{j=d-i+2}^{d}\alpha_{j}\frac{\lambda_{j}(x^*)}{\lambda_{d-i+1}(x^*)}\left(\frac{1-\gamma\lambda_{j}(x^*)}{1-\gamma\lambda_{d-i+1}(x^*)}\right)^{k}v_{j}(x^*)\right|\\
 \leq &\frac{2|y|}{|\alpha_{d-i+1}|}\left(\frac{\lambda_{d}(x^*)}{\lambda_{d-i+1}(x^*)}\right)\kappa_{{\rm dir},i}(\gamma)^{k}.
\end{split} 
\end{equation}
\normalsize
This ends the proof. 
\end{proof*}

\begin{remark}[Link with the power method] One can easily recover estimate~\eqref{eq:estimate_secant_Psi_discrete} by interpreting the sequence of directions $(\rho(\Psi_{k}(y)))_{k\in\NN}$ as the iterates of the power method to approximate $v_{d-i+1}(x^*)$ (see, e.g.,~\cite[Theorem 10.3.1]{MR2365296}). 
\end{remark}
The following result deals with the asymptotic directional behavior of the discrete flow $\Phi$.

\begin{theorem}[Directional convergence of the gradient method]
\label{HG_cas_discret} 
Under  Assumptions \ref{ass:H} and \ref{ass:min}:\\
(i) {\rm ($C^1$ conjugacy)} For $\eta>0$ small enough  there exists a local $C^1$-diffeomorphism $H\colon B(x^*,\eta)\to \RR^{d}$ satisfying $H(x^*)=0$ and $DH(x^*)=I_{d}$, such that, if $\gamma\in (0,2/\lambda_{d}(x^*))$,
\be
H(\Phi_{k}(x))= \Psi_{k}(H(x))\;\text{ for all }x\text{ in }B(x^*,\eta) \text{ and } k\geq 0.
\label{eq:conjugaison_cas_discret}
\ee
(ii) {\rm (Partition into stable manifolds)} 
Setting ${\mathcal S}_i^{H}=H^{-1}({\mathcal S}_i)$ for $i=0,\hdots,d$, the family $\{{\mathcal S}_i^{H}\}_{i=0}^d$ forms a partition of $B(x^*,\eta)$ into  manifolds such that, for all $i=0,\hdots,d$, $\dim {\mathcal S}_i^{H}=i$, and for all $k\ge 0$,
\[
\Phi_k({\mathcal S}_i^{H})\subset {\mathcal S}_i^{H}\quad \text{whenever } \,\gamma\neq \frac{1}{\lambda_{d-i+1}(x^*)}.
\]
(iii) {\rm (Convergence and directional alignment)} There exists $c>0$ such that, for all $x\in {\mathcal S}_i^{H}$ and $k\ge 0$,
\begin{equation}
\left|\Phi_k(x)-x^*\right|\le c\,|H(x)|\,\kappa_i(\gamma)^k,
\label{eq:Phi_exp_conv_discret}
\end{equation}
where $\kappa_i$ is defined in~\eqref{eq:def_kappa_gamma}. In addition, if $\gamma\in(0,2/(\lambda_{d-i+1}(x^*)+\lambda_{d}(x^*)))$, then, for every $x\in {\mathcal S}_{i}^{H}$, 
\be
\lim_{k\to\infty}\rho\left(\Phi_{k}(x)-x^*\right)=-\lim_{k\to \infty} \rho\left(\Phi_{k+1}(x)-\Phi_{k}(x)\right)
 \in \{-v_{d-i+1}(x^*), v_{d-i+1}(x^*)\}.
\label{eq:egalite_de_limites_cas_discret}
\ee
\end{theorem}

\begin{proof*} Ideas are similar to \Cref{HG_cas_continu}. 
By~\cite[Theorem (I)]{MR141856} there exist $\eta>0$ and a local diffeomorphism $H\colon B(x^*,\eta)\to\RR^{d}$ of class $C^{1}$ such that $H(x^*)=0$, $DH(x^*)=I_{d}$, and $H(\Phi_{1}(x))= \Psi_1(H(x))$ for all $x\in B(x^*,\eta)$. Since $\gamma<2/\lambda_{d}(x^*)$ and $\lambda_{d}$ is continuous in $B(x^*,\eta)$, by reducing $\eta$, if necessary, we can assume that $\gamma<2/L_{2}$, where $L_2$ satisfies~\eqref{eq:bounded_hessian}. Thus, Lemma~\ref{lem:stability_discrete_flow} implies that, for every $x\in B(x^*,\eta)$, we have $\Phi_{k}(x)\in B(x^*,\eta)$ for all $k\geq 0$. Using that $\Phi_{k_{1}+k_{2}}=(\Phi_{k_{1}}\circ\Phi_{k_2})(x)$ for all $x\in B(x^*,\eta)$ and $k_1$, $k_2\geq 0$, we deduce that the conjugacy~\eqref{eq:conjugaison_cas_discret} of (i) holds. 

Since~\eqref{eq:quadratic_steepest_descent_iterates} implies that for every $i=1,\hdots,d$,  $k\geq 0$, and $\gamma\in (0,2/\lambda_{d}(x^*))\setminus\{1/\lambda_{d-i+1}(x^*)\}$, $\Psi_{k}({\mathcal S}_{i})\subset {\mathcal S}_{i}$, it follows from~\eqref{eq:conjugaison_cas_discret} that $\Phi_{k}({\mathcal S}_{i}^{H})\subset {\mathcal S}_{i}^{H}$. Moreover, for every $k\geq 0$ and $x\in B(x^*,\eta)$, we have $\Phi_k(x)-x^*=H^{-1}\left(\Psi_k(H(x))\right)-H^{-1}(0)$ and, hence,~\eqref{eq:Phi_exp_conv_discret} follows from the boundedness of $DH^{-1}$ and~\eqref{eq:borne_Psi_k_cas_discret}.

Let us establish (ii). In order to show~\eqref{eq:egalite_de_limites_cas_discret}, let $x\in {\mathcal S}_{i}^{H}$ and, for every $j=1,\hdots,d$, set $\alpha_{j}:=\langle H(x),v_{j}(x^*)\rangle$. Since $\gamma\in(0,2/(\lambda_{d-i+1}(x^*)+\lambda_{d}(x^*)))$, one has, by definition,  $\kappa_{i}(\gamma)=1-\gamma\lambda_{d-i+1}(x^*)$. For every $k\geq 0$, we have
\begin{equation}
\begin{split}
& \; \left|\rho\left(\Phi_{k}(x)-x^{*}\right)-\text{{\rm sign}}(\alpha_{d-i+1})v_{d-i+1}(x^*)\right|\\
\leq & \; \left|\rho\left(\Phi_{k}(x)-x^{*}\right)-\rho\left(\Psi_{k}(H(x))\right)\right|+\left|\rho\left(\Psi_{k}(H(x))\right)-\text{{\rm sign}}(\alpha_{d-i+1})v_{d-i+1}(x^*)\right|.
\label{eq:triangular_inequality_1}
\end{split}
\end{equation}
Note that~\eqref{eq:quadratic_steepest_descent_iterates} and the orthogonality of the family $\{v_{i}(x^*)\}_{i=1}^{d}$ imply that $\left|\kappa_{i}^{-k}(\gamma)\Psi_k(H(x))\right|\geq |\alpha_{d-i+1}|>0$. In turn,~\eqref{eq:inegalite_basique_normalisee} yields 
\begin{equation}
\begin{split}
& \;\left|\rho\left(\Phi_{k}(x)-x^{*}\right)-\rho\left(\Psi_{k}(H(x))\right)\right|\\
= & \; \left|\rho\left(\kappa_{i}^{-k}(\gamma)\left(\Phi_{k}(x)-x^{*}\right)\right)-\rho\left(\kappa_{i}^{-k}(\gamma)\Psi_{k}(H(x))\right)\right|\\
\leq & \; \frac{2}{|\alpha_{d-i+1}|}\left|\kappa_{i}^{-k}(\gamma)\left[\Phi_{k}(x)-x^*-\Psi_{k}(H(x))\right]\right|.
\label{eq:est_rho_phi_k_discrete_1_bis}
\end{split}
\end{equation}
Since $\Phi_k(x)-x^*=H^{-1}(\Psi_{k}(H(x)))-H^{-1}(0)$ and $DH^{-1}(0)=I_{d}$, we also have
\begin{equation}
\Phi_k(x)-x^*-\Psi_{k}(H(x))=\left(\int_{0}^{1}\left(DH^{-1}(\tau \Psi_k(H(x)))-DH^{-1}(0)\right)\dd \tau\right) \Psi_k(H(x)),
\label{eq:difference_Phi_Psi_discrete}
\end{equation}
which, together with~\eqref{eq:est_rho_phi_k_discrete_1_bis},~\eqref{eq:borne_Psi_k_cas_discret}, and the continuity of $DH^{-1}$, yields 
\begin{equation}
\begin{split}
& \; \left|\rho\left(\Phi_{k}(x)-x^{*}\right)-\rho\left(\Psi_{k}(H(x))\right)\right|
\\
\leq & \; \frac{2|H(x)|}{|\alpha_{d-i+1}|}\int_{0}^{1}\left\|DH^{-1}(\tau \Psi_k(H(x)))-DH^{-1}(0)\right\|\dd \tau\xrightarrow[k\to\infty]{}0.
\label{eq:difference_Phi_Psi_discrete_bis}
\end{split}
\end{equation}
It then follows from~\eqref{eq:triangular_inequality_1} and Lemma~\ref{lem:directional_behavior_linear_flow_discrete} that 
\begin{equation}
\left|\rho\left(\Phi_{k}(x)-x^{*}\right)-\text{{\rm sign}}(\alpha_{d-i+1})v_{d-i+1}(x^*)\right|\xrightarrow[k\to\infty]{}0.
\label{eq:first_directional_convergence_discrete}
\end{equation}

Similarly, one has 
\begin{equation}
\begin{split}
& \; \left|\rho\left(\frac{\Phi_{k+1}(x)-\Phi_{k}(x)}{\gamma}\right)+\text{{\rm sign}}(\alpha_{d-i+1})v_{d-i+1}(x^*)\right|\\[6pt]
\leq & \; \left|\rho\left(\frac{\Phi_{k+1}(x)-\Phi_{k}(x)}{\gamma}\right)-\rho\left(\frac{\Psi_{k+1}(H(x))-\Psi_{k}(H(x))}{\gamma}\right)\right|\\[6pt]
+ & \; \left|\rho\left(\frac{\Psi_{k+1}(H(x))-\Psi_{k}(H(x))}{\gamma}\right)+\text{{\rm sign}}(\alpha_{d-i+1})v_{d-i+1}(x^*)\right|.
\label{eq:first_decomposition_rho_difference_Phi}
\end{split}
\end{equation}
From~\eqref{eq:calcul_difference_psi} and the orthogonality of the family $\{v_{i}(x^*)\}_{i=1}^{d}$, we have
 $$
 \left| \lambda_{d-i+1}(x^*)^{-1}\kappa_{i}^{-k}(\gamma)\left(\frac{\Psi_{k+1}(H(x))-\Psi_{k}(H(x))}{\gamma}\right)\right|\geq |\alpha_{d-i+1}|>0,
 $$ 
which, by~\eqref{eq:inegalite_basique_normalisee}, implies after proper rescaling
\small
\begin{equation}
\begin{split}
 & \; \left|\rho\left(\frac{\Phi_{k+1}(x)-\Phi_{k}(x)}{\gamma}\right)-\rho\left(\frac{\Psi_{k+1}(H(x))-\Psi_{k}(H(x))}{\gamma}\right)\right|\\
 \leq & \; \frac{2\kappa_{i}^{-k}(\gamma)}{\lambda_{d-i+1}(x^*)|\alpha_{d-i+1}|}\left|\frac{\Phi_{k+1}(x)-\Phi_{k}(x)}{\gamma}-\frac{\Psi_{k+1}(H(x))-\Psi_{k}(H(x))}{\gamma}\right|.
 \label{eq:first_decomposition_rho_difference_Psi}
\end{split}
\end{equation}
\normalsize
Note that, by~\eqref{def:steepest_descent_iterates} and~\eqref{def:steepest_descent_iterates_quadratic_case}, we have
$$
\left|\frac{\Phi_{k+1}(x)-\Phi_{k}(x)}{\gamma}-\frac{\Psi_{k+1}(H(x))-\Psi_{k}(H(x))}{\gamma}\right|=\left|\nabla f(\Phi_{k}(x))-\nabla^{2}f(x^*)\Psi_{k}(H(x))\right|.
$$
Since $f$ is of class $C^{3}$ and $\nabla f(x^*)=0$, there exists $c>0$ such that 
$$
|\nabla f(\Phi_{k}(x))-\nabla^{2}f(x^*)(\Phi_k(x)-x^*)|\leq c|\Phi_k(x)-x^*|^{2}
$$
and, hence, by the triangular inequality,~\eqref{eq:bounded_hessian}, and~\eqref{eq:difference_Phi_Psi_discrete},
\begin{equation}
\begin{split}
& \; \left|\frac{\Phi_{k+1}(x)-\Phi_{k}(x)}{\gamma}-\frac{\Psi_{k+1}(H(x))-\Psi_{k}(H(x))}{\gamma}\right|\\
\leq & \;
c|\Phi_k(x)-x^*|^{2}+L_{2}\left|\Phi_k(x)-x^*-\Psi_{k}(H(x))\right|\\
\leq & \;  c|\Phi_k(x)-x^*|^{2}+L_2\int_{0}^{1}\left\|DH^{-1}(\tau \Psi_k(H(x)))-DH^{-1}(0)\right\|\dd \tau\,|\Psi_k(H(x))|.
\label{eq:estimation_one_term_triangular_ineq}
\end{split}
\end{equation}
It then follows from~\eqref{eq:first_decomposition_rho_difference_Phi},~\eqref{eq:first_decomposition_rho_difference_Psi},~\eqref{eq:Phi_exp_conv_discret},~\eqref{eq:borne_Psi_k_cas_discret}, the continuity of $DH^{-1}$, and Lemma~\ref{lem:directional_behavior_linear_flow_discrete} that
\begin{equation}
\left|\rho\left(\Phi_{k+1}(x)-\Phi_{k}(x)\right)+\text{{\rm sign}}(\alpha_{d-i+1})v_{d-i+1}(x^*)\right|\xrightarrow[k\to\infty]{}0.
\label{eq:second_directional_convergence_discrete}
\end{equation}
Therefore,~\eqref{eq:egalite_de_limites_cas_discret} follows from~\eqref{eq:first_directional_convergence_discrete} and~\eqref{eq:second_directional_convergence_discrete}.
\end{proof*}

\begin{remark}[The strict stepsize condition is sharp]
\label{rem:speed_of_convergence_discrete_case}
Indeed, in the quadratic case, if  $\gamma\in [2/(\lambda_{d-i+1}(x^*)+\lambda_{d}(x^*)),2/\lambda_{d})$ and $x\in {\mathcal S}_{i}^{H}$, the directions  $\rho\left(\Psi_{k}(H(x))\right)$ and $\rho\left(\Psi_{k+1}(H(x))-\Psi_{k}(H(x))\right)$ do not converge in general. 

For instance, in the case of $\left(\rho\left(\Psi_{k}(H(x))\right)\right)_{k\in\NN}$, if $\gamma>2/(\lambda_{d-i+1}(x^*)+\lambda_{d}(x^*))$ and $\alpha_{d}=\langle H(x),v_{d}(x^*)\rangle\neq 0$, arguing as in the proof of Lemma~\ref{def:kappa_dir}, one gets that 
$$
\begin{aligned}
\rho\left(\Psi_{2k}(H(x))\right)&\xrightarrow[k\to\infty]{} \text{{\rm sign}}(\alpha_{d})v_{d}(x^*),\\
\rho\left(\Psi_{2k+1}(H(x))\right)&\xrightarrow[k\to\infty]{}-\text{{\rm sign}}(\alpha_{d})v_{d}(x^*),
\end{aligned}
$$
i.e., $\rho(\Psi_{k}(x))$ aligns, in an oscillatory manner, with $\RR v_{d}(x^*)$, the eigenspace associated with the largest eigenvalue $\lambda_{d}(x^*)$ of $\nabla^{2} f(x^*)$. If $\gamma=2/(\lambda_{d-i+1}(x^*)+\lambda_{d}(x^*))$, which is the stepsize for which one has the optimal rate of convergence in~\eqref{eq:borne_Psi_k_cas_discret}, then 
$$
\begin{aligned}
\rho\left(\Psi_{2k}(H(x))\right)&\xrightarrow[k\to\infty]{}\frac{\alpha_{d-i+1}v_{d-i+1}(x^*)+\alpha_{d}v_{d}(x^*)}{\sqrt{\alpha_{d-i+1}^{2}+\alpha_{d}^2}},\\[4pt]
\rho\left(\Psi_{2k+1}(H(x))\right)&\xrightarrow[k\to\infty]{}\frac{\alpha_{d-i+1}v_{d-i+1}(x^*)-\alpha_{d}v_{d}(x^*)}{\sqrt{\alpha_{d-i+1}^{2}+\alpha_{d}^2}},
\end{aligned}
$$
i.e. $\rho(\Psi_{k}(x))$  oscillates between two elements of $\text{span}\{v_{d-i+1}(x^*),v_{d}(x^*)\}$. The same behavior occurs for $\rho\left(\Psi_{k+1}(H(x))-\Psi_{k}(H(x))\right)$.
The sequences $\rho\left(\Phi_{k}(x)\right)$ and $\rho\left(\Phi_{k+1}(x)-\Phi_{k}(x)\right)$ can be studied as in the proof of Theorem~\ref{HG_cas_discret}.

\end{remark}

\begin{theorem}
\label{th:align_discrete}

(i) {\rm (Generic alignment)} Suppose  that Assumption \ref{ass:H} holds at each critical point $x^*$ and that $\nabla f$ is $L$-Lipschitz continuous with $L>0$. If $\gamma\in (0,1/L)$ and $f$ is coercive, the set of initial conditions for which the gradient method converges to local minimizers of $f$, with both velocities and secants aligning with the associated talweg, has full Lebesgue measure.

(ii) {\rm (Alignment rate)} Suppose that $f$ is $C^\infty$, let $x$ be such that  $\Phi_k(x)\to x^*$ as $k\to\infty$, where $x^*$ is a minimizer as in \Cref{ass:min} and $\Phi_k(x)\in  {\mathcal S}_d^H(x^*)$ for some $k$. Assume in addition the non-resonance conditions at $x^*$: 
for all integers $m_1,\dots,m_d$ such that $\sum_{j=1}^d m_j \ge 2$,
\[\lambda_i(x^*) \neq \sum_{j=1}^d m_j \lambda_j(x^*) \text{ for all } i=1,\dots,d.
\qquad
\]

If
$\displaystyle
\gamma \in \left(0,\frac{2}{\lambda_1(x^*)+\lambda_2(x^*)}\right),
$ 
then alignment with the tangent to the talweg goes at a rate:
\small
\begin{align*}
\dist \!\Big(
\rho(\Phi_k(x)-x^*),
\,T_{x^*}\tal([r^*,\bar r))
\Big)
&=
\mathrm{O}\!\left(
\left(\frac{1-\gamma\lambda_2(x^*)}{1-\gamma\lambda_1(x^*)}\right)^k
\right)
+\mathrm{O}\!\left((1-\gamma\lambda_1(x^*))^k\right),\\
\dist \!\Big(
\rho\!\left(\Phi_{k+1}(x)-\Phi_k(x)\right),
\,T_{x^*}\tal([r^*,\bar r))\Big)
&=
\mathrm{O}\!\left(
\left(\frac{1-\gamma\lambda_2(x^*)}{1-\gamma\lambda_1(x^*)}\right)^k
\right)
+\mathrm{O}\!\left((1-\gamma\lambda_1(x^*))^k\right). 
\end{align*}
\normalsize
\end{theorem}

\begin{proof}
Item (i) is a mere adaptation of the proof of \Cref{th:align}~(i), using \Cref{HG_cas_discret}, and the fact that, for every critical point $x^*$,
$$
\frac{2}{\lambda_1(x^*)+\lambda_d(x^*)}\geq 1/L.
$$   
We prove (ii). Observe that, by assumption $\Phi_k(x)\neq\Phi_{k+1}(x)$ for $k\geq 0$, both being different from $x^*$. Under the non-resonance condition and the smoothness of $f$,
Sternberg's theorem \cite[Theorem 2]{MR96853} ensures that the conjugacy
$H$ may be chosen $C^k$ with $k\ge2$, so that 
there exists $C_H>0$ such that
\[
\|DH^{-1}(y)-DH^{-1}(0)\|
\le C_H |y| \mbox{ for $y$ in a neighborhood of $0$.}
\]
 Write $
H(x)=\sum_{j=1}^d \alpha_j v_j(x^*),
\mbox{ with }
\alpha_1\neq0.$

\smallskip
\noindent
Let us estimate for $\rho(\Phi_k(x)-x^*)$. Using the same triangular decomposition as in the proof of 
Theorem~\ref{HG_cas_discret},
\begin{align*}
& \;\left|
\rho(\Phi_k(x)-x^*)
-
\text{sign}(\alpha_1)v_1(x^*)
\right| \\
\le & \;
\left|
\rho(\Phi_k(x)-x^*)
-
\rho(\Psi_k(H(x)))
\right|
+
\left|
\rho(\Psi_k(H(x)))
-
\text{sign}(\alpha_1)v_1(x^*)
\right|.
\end{align*}
The second term is controlled by 
Lemma~\ref{lem:directional_behavior_linear_flow_discrete}: $\displaystyle
\left|
\rho(\Psi_k(H(x)))
-
\text{sign}(\alpha_1)v_1(x^*)
\right|
=
\mathrm{O}\!\left(\kappa_{\mathrm{dir},1}(\gamma)^k\right).$
 For the first term, using 
\eqref{eq:difference_Phi_Psi_discrete_bis} and \eqref{eq:borne_Psi_k_cas_discret}
together with the Lipschitz continuity of $DH^{-1}$ yields
\[
\left|
\rho(\Phi_k(x)-x^*)
-
\rho(\Psi_k(H(x)))
\right|
=
\mathrm{O}\!\left(\kappa_1(\gamma)^k\right).
\]
  Recalling that $\mathrm{sign}(\alpha_1)v_1(x^*)\in T_{x^*}\tal([r^*,\bar r))$ 
 proves the secant estimate.
\medskip
\noindent
Let us estimate now the direction of the ``discrete velocity". Arguing as in 
\eqref{eq:first_decomposition_rho_difference_Phi},
\begin{align*}
&\; \left|
\rho\!\left(\frac{\Phi_{k+1}(x)-\Phi_k(x)}{\gamma}\right)
+
\text{sign}(\alpha_1)v_1(x^*)
\right|\\
\le&\;
\left|
\rho\!\left(\frac{\Phi_{k+1}(x)-\Phi_k(x)}{\gamma}\right)
-
\rho\!\left(\frac{\Psi_{k+1}(H(x))-\Psi_k(H(x))}{\gamma}\right)
\right|\\
&
+
\left|
\rho\!\left(\frac{\Psi_{k+1}(H(x))-\Psi_k(H(x))}{\gamma}\right)
+
\text{sign}(\alpha_1)v_1(x^*)
\right|.
\end{align*}

The second term is again controlled by 
the quadratic case Lemma~\ref{lem:directional_behavior_linear_flow_discrete}. The first term is estimated using 
\eqref{eq:first_decomposition_rho_difference_Psi}, 
\eqref{eq:estimation_one_term_triangular_ineq},
\eqref{eq:Phi_exp_conv_discret},
\eqref{eq:borne_Psi_k_cas_discret},
and the Lipschitz property of $DH^{-1}$,
yielding a term of order $\mathrm{O}(\kappa_1(\gamma)^k)$. Combining the above estimates concludes the proof.
\end{proof}

\begin{remark}[Sharpness of the alignment rate in the discrete case]
\label{rem:alignment_speed_sharpness_discrete}
Unlike in the quadratic case, considered in \Cref{lem:directional_behavior_linear_flow_discrete},
one cannot, in general, expect a pure spectral-gap rate.  Indeed, as in \Cref{rem:alignement-rate-cont}, consider
$f(x_1,x_2)=\frac12\lambda_1 x_1^2+\frac12\lambda_2 x_2^2+a x_1^2 x_2$ with $0<\lambda_1<\lambda_2$ and $a\neq 0$. Assume that $\lambda_{2}>2\lambda_{1}$ and let $\gamma\in (0,1/\lambda_{2})$. Given $x_0=(\alpha_1,\alpha_2)$, with $\alpha_{1}\neq 0$, we consider the gradient descent iterates $x_{k}=(x_{k,1},x_{k,2})$ defined through
\begin{equation}
\begin{aligned}
x_{1,k+1}
&=
(1-\gamma\lambda_1) x_{1,k}
-
2a\gamma x_{1,k}x_{2,k},
\\
x_{2,k+1}
&=
(1-\gamma\lambda_2) x_{2,k}
-
a\gamma x_{1,k}^2,
\end{aligned}
\label{eq:sys_discrete_eqns}
\end{equation}
for all $k\in\NN$. We show in Section~\ref{subsec:some_proof} in the Appendix that, provided that $|x_0|$ is small enough, 
\begin{equation}
\mathrm{dist}\!\left(
\rho(x_k),
\,T_0\extun
\right)
= C(1-\gamma\lambda_1)^k+\mathrm{o}\left((1-\gamma\lambda_1)^k\right), \text{ for some } C>0.
\label{eq:dist_talweg_ex_2}
\end{equation}
\end{remark}

\begin{corollary}[Valley is generically reached in finite time (discrete case)]\label{cor:valley_absortion_discrete_case}
 Suppose that Assumptions  \ref{ass:H} and \ref{ass:min} hold,  let $\eta>0$,  $H\colon B(x^*,\eta)\to\RR^{d}$ be as in Theorem~\ref{HG_cas_discret}{\rm(i)}, and let  $\gamma\in(0,2/(\lambda_{1}(x^*)+\lambda_{d}(x^*)))$. Then,  for every $w>0$ and $x\in {\mathcal S}_{d}^{H}(x^*)$, there exists $k_0\in\NN$ such that 
$$
\Phi_{k}(x)\in \vall_{f,w}(x^*)\quad\text{for all } k\geq k_0.
$$ 
\end{corollary}
\begin{proof*} Using directional convergences to the talweg and \Cref{valley_inclusions}, the proof is similar to that of \Cref{on_va_vers_V_eps_bis}.\end{proof*}

Let us conclude by a result on gradient iterates concentration  within valleys:
\begin{theorem}[Volume concentration in valleys for the gradient method] 
\label{rapport_de_volumes_cas_discret}	
Suppose that Assumptions  \ref{ass:H} and \ref{ass:min} hold. Then, if $\eta>0$ is  small enough and $\gamma\in(0,1/\lambda_{d}(x^*))$, for every measurable
set $S \subset  B(x^*,\eta)$, with $\vol(S)>0$, and $w>0$, we have
\begin{equation}
\lim_{k\to \infty}\frac{\vol\left(\Phi_{k}(S) \cap \vall_{f,w}(x^*)\right)}{\vol\left(\Phi_{k}(S)\right)}=1.  
\label{rapport_de_volumes_cas_discret_eq}	
\end{equation}
\end{theorem}
\begin{proof*}  
Since $\gamma<2/L_{2}$ for $\eta$ small enough, Lemma~\ref {lem:stability_discrete_flow} implies that, for every $k\geq 0$, $\Phi_{k}(B(x^*,\eta)) \subset B(x^*,\eta)$. Using that $\gamma\in(0,1/\lambda_{d}(x^*))$, one has that $\Phi_1\colon B(x^*,\eta)\to\RR^{d}$ is a local diffeomorphism, and hence $\Phi_k=\Phi_1\circ\dots\circ\Phi_1$ ($k\text{ times}$) is also a local diffeomorphism. Moreover, for every $x\in B(x^*,\eta)$ and $k\geq 1$, the chain rule yields
\begin{equation}
\begin{aligned}
\det(\partial_x\Phi_{k}(x))&=\det\left(\left(I_{d}-\gamma \nabla^{2}f\left(\Phi_{k-1}(x)\right)\right) \partial_x\Phi_{k-1}(x)\right)\\
&=\prod_{\ell=0}^{k-1}\det\left( I_{d}-\gamma \nabla^{2}f\left(\Phi_\ell(x)\right) \right).
\end{aligned}
\label{eq:proof_concentration_disc_flow_1}
\end{equation}
Set $\overline{p}:= \det\left( I_{d}-\gamma \nabla^{2}f\left(x^*\right) \right)=\prod_{i=1}^{d}(1-\gamma\lambda_{i}(x^*))>0$ and notice that,  since $\RR^{d}\ni x'\mapsto \det\left( I_{d}-\gamma \nabla^{2}f\left(x'\right) \right)\in\RR$ is of class $C^1$ and  $\Phi_{\ell}(B(x^*,\eta)) \subset B(x^*,\eta),$ for all $\ell\geq 0
$, there exists a constant $C_1>0$ such that 
\begin{equation}
\overline{p}-C_1|\Phi_\ell(x)-x^*|\leq \det\left( I_{d}-\gamma \nabla^{2}f\left(\Phi_\ell(x)\right) \right) 
\leq \overline{p}+C_1|\Phi_\ell(x)-x^*|\quad\text{for all }x\in B(x^*,\eta). 
\label{eq:proof_concentration_disc_flow_2}
\end{equation} 

Shrink $\eta$ so that $0<\eta<\overline{p}/C_1$, set $a=\overline{p}-C_1\eta>0$ and $b=\overline{p}+C_1\eta$. For every $x\in B(x^*,\eta)$ and $\ell\geq 0$, setting $p_{\ell}(x)= C_1|\Phi_\ell(x)-x^*|$, Lemma~\ref{lem:stability_discrete_flow} implies that $\overline{p}\pm p_{\ell}(x)\in [a,b]\subset(0,\infty)$ for all $\ell\geq 0$. It follows from~\eqref{eq:proof_concentration_disc_flow_1},~\eqref{eq:proof_concentration_disc_flow_2}, and the inequality $|\log(\alpha)-\log(\beta)|\leq |\alpha-\beta|/a$ for all $\alpha,\,\beta\in [a,b]$,  that, for every $x\in B(x^*,\eta)$ and $k\geq 0$, we have
\begin{equation}
\det(\partial_x\Phi_{k}(x) ) \geq \exp\left(\sum_{\ell=0}^{k-1}\log\left(\overline{p}-p_{\ell}(x)\right) \right) \geq \exp\left(k\log(\overline{p})-\frac{C_1}{a}\sum_{\ell=0}^{k-1} |\Phi_\ell(x)-x^*| \right)
\label{eq:proof_concentration_disc_flow_3}
\end{equation}
and 
\begin{equation}
\det(\partial_x\Phi_{k}(x)))\leq \exp\left(\sum_{\ell=0}^{k-1}\log\left(\overline{p}+p_{\ell}(x)\right) \right)\leq \exp\left(k\log(\overline{p})+\frac{C_1}{a}\sum_{\ell=0}^{k-1} |\Phi_\ell(x)-x^*| \right). 
\label{eq:proof_concentration_disc_flow_4}
\end{equation}

Let $H$ be given by Theorem~\ref{HG_cas_discret}. From~\eqref{eq:Phi_exp_conv_discret}, the fact that $H$ can be taken to be bounded, and~\eqref{eq:def_kappa_gamma}, we deduce that 
$$
\zeta:=\frac{C_1}{a}\sum_{\ell=0}^{\infty}\sup_{x\in B(x^*,\eta)}|\Phi_{\ell}(x)-x^*|<+\infty.
$$
In turn,~\eqref{eq:proof_concentration_disc_flow_3} and~\eqref{eq:proof_concentration_disc_flow_4} yield
\begin{equation}
\exp\left(k\log(\overline{p})-\zeta\right)\leq \det(\partial_x\Phi_{k}(x))\leq \exp\left(k\log(\overline{p})+\zeta\right),
\label{eq:proof_concentration_disc_flow_5}
\end{equation}
for all  $x\in B(x^*,\eta)$ and $k\geq 0$. Since $\Phi_{k}$ is a local diffeomorphism, if $S\subset B(x^*,\eta)$ is measurable, with $\vol(S)>0$,  one also has $\vol(\Phi_{k}(S))>0$. Moreover, for every $w>0$, we have
$$
\frac{\vol\left(\Phi_{k}(S) \cap \vall_{f,w}(x^*)\right)}{\vol\left(\Phi_{k}(S)\right)}=1- R_{k}\quad\text{for all }k\in\NN,
$$
where 
\small
$$
R_{k}:=\frac{\vol\left(\Phi_{k}(S) \setminus \vall_{f,w}(x^*)\right)}{\vol\left(\Phi_{k}(S)\right)}=\frac{\displaystyle\int_{\Phi_{k}(S) \setminus \vall_{f,w}(x^*)}\dd x}{\displaystyle\int_{\Phi_{k}(S)}\dd x}=\frac{\displaystyle\int_{S\setminus \Phi_{k}^{-1}(\vall_{f,w}(x^*))}\det(\partial_x\Phi_{k}(x))\dd x}{\displaystyle\int_{S}\det(\partial_x\Phi_{k}(x))\dd x},
$$
\normalsize
which, together with~\eqref{eq:proof_concentration_disc_flow_5}, implies that 
$$
R_k\leq \exp(2\zeta)\frac{\vol\left(S \setminus \Phi_{k}^{-1}(\vall_{f,w}(x^*))\right)}{\vol(S)}.
$$
In turn, Corollary~\ref{cor:valley_absortion_discrete_case}, $\vol\left(B(x^*,\eta)\setminus {\mathcal S}_{d}^{H}(x^*)\right)=0$, and Lebesgue dominated convergence imply that $R_{k}\to 0$, as $k\to\infty$, from which~\eqref{rapport_de_volumes_cas_discret_eq} follows.  
\end{proof*}

{\bf Acknowledgements:} 
PB, JB, TM thank TSE-P and acknowledge financial support of the ANR (Programmes d'Investissements d'Avenir CHESS ANR-17-EURE-0010).

FJS was partially supported by l'Agence
Nationale de la Recherche (ANR), project ANR-22-CE40-0010, by KAUST through the subaward
agreement ORA-2021-CRG10-4674.6, and by Minist\`ere de l'Europe et des Affaires \'etrang\`eres (MEAE), project MATH AmSud 23-MATH-17.

JB thanks AI Interdisciplinary Institute ANITI funding, through the ANR under the France 2030 program (grant ANR-23-IACL-0002), Chair TRIAL, ANR MAD, ANR Regulia, Air Force Office of Scientific Research, Air Force Material Command, USAF, under grant number FA8655-22-1-7012. 
%%%%%%%%%%%%%%%%%%%%%%%%%%%%%%%%%%%%%%%%%%%%%%%%%%%%%%%

\appendix
\section{Appendix}
%%%%%%%%%%%%%%%%%%%%%%%%%%%%%%%%%%%%%%%%%%%%%%%%%%%%%%%

\subsection{Additional facts}

\begin{proposition}[More properties of the talweg] Make Assumptions \ref{ass:H} and \ref{ass:min}. 
If $\bar r$ is sufficiently close to $r^*$, then for every $r\in (r^*,\bar r)$ the set $\tal (r)$ has exactly two elements:
\begin{equation}
\tal(r)=\{\theta^{-}(r),\theta^{+}(r)\}\;\;\text{and they satisfy }\;\langle\theta^{\pm}(r)-x^*,\pm v_{1}(x^*)\rangle>0.
\label{eq:talweg_two_elements}
\end{equation}
Moreover, the curves $(r^*,\bar r)\ni r\to\theta^{\pm}(r)\in\RR^{d}$ are of class $C^{1}$ and, denoting by $\dot{\theta}^{\pm}$ their derivatives, we have $\lim_{r\to f(x^*),\, r>f(x^*)}\theta^{\pm}(r)=x^*$, 
\begin{equation}
\lim_{r\to f(x^*),\, r>f(x^*)}\rho\left(\dot{\theta}^{\pm}(r)\right)=\pm v_{1}(x^*),\quad\text{and}\quad\lim_{r\to f(x^*),\, r>f(x^*)}|\dot{\theta}^{\pm}(r)||\nabla f(\theta^{\pm}(r))|=1. 
\label{eq:behavior_theta_dot}
\end{equation}
\label{th:study_of_theta}
\end{proposition}
\begin{proof*} By Theorem~\ref{th:talweg_is_gradient_extremal}, for $\eta>0$ small enough, there exist $\bar{t}>0$ and a $C^{1}$ and a curve $\gamma\colon(-\bar{t},\bar{t})\to E_{1}\cap [f<\bar r]_{x^*}$ satisfying $\gamma(0)=x^*$ and $\dot{\gamma}(0)=v_{1}(x^*)$. Hence, $\gamma(t)=x^*+tv_{1}(x^*)+\mathrm o(t)$ for $t\in (-\bar{t},\bar{t})$. Let us define $\phi(t):=f(\gamma(t))$ on $(-\bar{t},\bar{t})$. Since $\nabla f(\gamma(t))=\nabla^{2}f(x^*)(\gamma(t)-x^*)+\mathrm o(t)$ and $\dot{\phi}(t)=\langle \nabla f(\gamma(t)),\dot{\gamma}(t)\rangle=\langle \nabla f(\gamma(t)),v_{1}(x^*)+\mathrm o(1)\rangle$, we deduce that 
$$
\dot{\phi}(t)=\left\langle\nabla^{2}f(x^*)(\gamma(t)-x^*)+\mathrm o(t),v_{1}(x^*)+\mathrm o(1)\right\rangle=\lambda_{1}(x^*)t+\mathrm o(t),\quad\text{for all }t\in(-\bar{t},\bar{t}).
$$
Reducing $\bar{t}$, if necessary, we have $\dot{\phi}<0$ in $(-\bar{t},0)$ and  $\dot{\phi}>0$ in $(0,\bar{t})$. In particular, for every $r\in\left(f(x^*),\min\{\lim_{t\to-\bar{t}}\phi(t),\lim_{t\to\bar{t}}\phi(t)\}\right)$ one has the existence of $t^{-}\in(-\bar{t},0)$ and $t^{+}\in(0,\bar{t})$ such that $\theta^{-}(r):=\gamma(t^{-})$ and $\theta^{+}(r):=\gamma(t^+)$ satisfy $\tal(r)=\{\theta^{-}(r),\theta^{+}(r)\}$. Notice that $\langle \theta^{+}(r)-x^*,v_{1}(x^*)\rangle=t^++\mathrm o(t^{+})$ and $\langle \theta^{-}(r)-x^*,v_{1}(x^*)\rangle=t^{-}+\mathrm o(t^{-})$, and hence, reducing $\bar{t}$ if necessary, we get $\langle \theta^{+}(r)-x^*,v_{1}(x^*)\rangle>0$ and $\langle \theta^{-}(r)-x^*,v_{1}(x^*)\rangle<0$. In conclusion, if $\bar r-r^*>0$ is small enough,~\eqref{eq:talweg_two_elements} holds. 

Define $F\colon B(x^*,\eta)\times\RR\times (r^*,\bar r)\to \RR^d \times \RR$ as 
$$
F(x,\lambda,r)=\left(\nabla^{2}f(x) \nabla f(x) -\lambda   \nabla f(x ) , f(x)- r\right).
$$
Setting $(\theta^{\pm}_{r},\lambda_{1,r}^{\pm}):=(\theta^{\pm}(r),\lambda_{1}(\theta^{\pm}(r)))$, for all $r\in (r^*,\bar r)$, one has 
$$
F(\theta^{\pm}_{r},\lambda_{1,r}^{\pm},r)=0.
$$
In view of applying the implicit function theorem,  we consider the linearized system of equations, with unknowns $(\delta_1,\delta_2)$  in $\RR^d\times\RR$, $(\theta', \lambda')$ in $\RR^d \times \RR$,
\be
\begin{split}
&\; (\delta_1, \delta_2) \\
=& \; D_{(x,\lambda)} F(\theta^{\pm}_{r}, \lambda_{1,r}^{\pm},r).(\theta',\lambda')\\
=& \left( D^3f(\theta^{\pm}_{r})\nabla f(\theta^{\pm}_{r}) \theta'+(\nabla^2f(\theta^{\pm}_{r}))^2\theta'  - \lambda_{1,r}^{\pm} \nabla^{2}f(\theta^{\pm}_{r}) \theta' - \lambda'\nabla f(\theta^{\pm}_{r}), \nabla f(\theta^{\pm}_{r}) \cdot \theta'  \right).
\end{split}
\label{inversibility_of_the_jacobian}
\ee
Back to the proof of Theorem~\ref{th:talweg_is_gradient_extremal} (see~\eqref{second_order_condition_sufficient_talweg_proof_lag}), if $\bar r-r^*>0$ is small enough, we have 
$$
 \left\langle  \nabla^2_{xx} L_r(\theta^{\pm}_{r}, \lambda_{1,r}^{\pm})h, h\right\rangle>0,\quad\text{for all $h\in\RR^d\setminus\{0\}$ such that } \langle\nabla f(\theta^{\pm}_{r}),h\rangle=0,
$$
which implies that
\be\label{auxiliary_optimization_problem}\ba{l}
\min\limits_{\theta'\in\RR^{d}} \;  \displaystyle\frac12\left\langle  \nabla^2_{xx} L_r(\theta^{\pm}_{r}, \lambda_{1,r}^{\pm})\theta', \theta' \right\rangle - \delta_1 \cdot \theta' \\[10pt]
\mbox{s.t. } \hspace{0.8cm}  \langle\nabla f(\theta^{\pm}_{r}),\theta'\rangle = \delta_2.
\ea
\ee
admits a unique solution $\hat{\theta}$ and, thanks to~\eqref{eq:non_null_gradient}, a unique Lagrange multiplier $\hat{\lambda}$. Therefore, by the implicit function theorem, the curves $ r\to \theta_{r}^{\pm}$ and $ r\to \lambda_{1,r}^{\pm}$ are  $C^{1}$ on $(r^*,\bar r)$, with derivatives $\dot{\theta}_{r}^{\pm}$ and $\dot{\lambda}_{r}^{\pm}$ satisfying 
\be
\ba{rcl}
 \left(D^{3}f(\theta_{r}^{\pm}) \nabla f(\theta_{r}^{\pm})  + \left(\nabla^2f(\theta_{r}^{\pm})\right)^2-\lambda_{1,r}^{\pm}\nabla^2 f(\theta_{r}^{\pm}) \right)\dot{\theta}_{r}^{\pm} &=&  \dot{\lambda}_{r}^{\pm} \nabla f(\theta_{r}^{\pm}), \\[6pt]
\langle\nabla f(\theta_{r}^{\pm}), \dot{\theta}_{r}^{\pm}\rangle &=& 1.
\ea
\label{characterization_of_the_derivatives}
\ee 
Notice that, by~\eqref{quadratic_growth} and the continuity of $\lambda_{1}$, $\theta_{r}^{\pm}\to x^*$. As a consequence, since $\rho(\dot{\theta}_{r}^{\pm})\in T_{\theta_{r}^{\pm}}E_{1}$, one deduces that, as $r\to r^*$, with $r>r^*$, any cluster point of $\rho(\dot{\theta}_{r}^{\pm})$ belongs to $T_{x^*}E_{1}=\RR v_{1}(x^*)$ and hence must belong to $\{-v_{1}(x^*),v_{1}(x^*)\}$. On the other hand by Theorem~\ref{th:talweg_is_gradient_extremal}, the gradient is an eigenvector along the talweg thus $\rho\left(\nabla f(\theta_{r}^{\pm})\right)={\mathcal S}_{r}^{\pm} v_{1}(\theta_{r}^{\pm})$, with ${\mathcal S}_{r}^{\pm}\in\{-1,1\}$. It then follows from  the second equation in~\eqref{characterization_of_the_derivatives} that $\langle v_{1}(\theta_{r}^{\pm}),\dot{\theta}_{r}^{\pm}\rangle=\pm 1$ and hence $\rho(\dot{\theta}_{r}^{\pm})$ has only one cluster point given by $\pm v_{1}(x^*)$, from which the first equality in~\eqref{eq:behavior_theta_dot} follows. Finally, from the second equation in~\eqref{characterization_of_the_derivatives}, we also have
$$
1=\langle v_{1}(x^*),v_{1}(x^*)\rangle=\lim_{r\to r^*,\, r>r^*}\left\langle \frac{\nabla f(\theta_{r}^{\pm})}{|\nabla f(\theta_{r}^{\pm})|},\frac{\dot{\theta}_{r}^{\pm}}{|\dot{\theta}_{r}^{\pm}|}\right\rangle=\lim_{r\to r^*,\, r>r^*}\frac{1}{|\nabla f(\theta_{r}^{\pm})||\dot{\theta}_{r}^{\pm}|}, 
$$
from which the second equality in~\eqref{eq:behavior_theta_dot} follows. 
\end{proof*}

\begin{remark}
As a consequence of the second relation in~\eqref{eq:behavior_theta_dot}, one can provide an estimate for the blow-up of $|\dot{\theta}^{\pm}(r)|$ as $r\to f(x^*)$, with $r>f(x^*)$. Indeed, since $\nabla f(\theta^{\pm}(r))=\nabla^{2}f(x^*)(\theta^{\pm}(r)-x^*)+\mathrm o(|\theta^{\pm}(r)-x^*|)$, if $\eta$ is small enough,~\eqref{eq:behavior_theta_dot} and~\eqref{quadratic_growth} imply that 
$$
|\dot{\theta}^{\pm}(r)|\geq \frac{1}{2\lambda_{d}(x^*)|\theta^{\pm}(r)-x^*|}\geq \frac{1}{2\lambda_{d}(x^*)\sqrt{r-r^*}},\quad\text{for all }r\in (r^*,\bar r).
$$
Observe that this estimate can be combined with the results of \cite{bolte2010characterizations}, to show the existence of $c>0$ such that
$$|\nabla \sqrt{f-r^*}|(x)\geq c \;\;\mbox{ over \;\; $[f<\bar r]_{x^*}\setminus\{x^*\}$.}$$
\end{remark}
\medskip

\begin{lemma}[Balls are absorbing] \label{lem_Step_1}
Under \Cref{ass:min}, if $\eta>0$ is small enough, then $\Phi_t(B(x^*,\eta)) \subset B(x^*,\eta),$ for all $t\in [0,\infty).$
\end{lemma}
\begin{proof*}
Just observe that if $\eta>0$ is small enough $f\colon B(x^*,\eta)\to \RR$ is convex, as $\lambda_1>0$ on  $B(x^*,\eta)$. It suffices then observe that $|\Phi_t(x)-x^*|$ is decreasing as $\displaystyle
   \frac{d}{d t}|\Phi_t(x)-x^*|^2= 2\langle \nabla f(x^*) - \nabla f(\Phi_t(x)), \Phi_t(x)-x^* \rangle \leq  0.$ 
\end{proof*}

Denote by $\D_d(\RR)$ (respectively, $\mathrm{O}_d(\RR),$ $\mathcal{S}_d(\RR))$ the set of diagonal (respectively, orthogonal, symmetric) matrices. We have: 

\begin{lemma}[Parametric diagonalization]
\label{lemdiff}
Let $F\colon\mathrm{O}_d(\RR)\times\D_d(\RR)\longrightarrow\mathcal{S}_d(\RR)$ be defined by $F(P,\mathscr{D})= P\mathscr{D}P^\top$ and let $(P_0,\mathscr{D}_0)\in\mathrm{O}_d(\RR)\times\D_d(\RR)$ be such that for any $i\neq j,$ $(\mathscr{D}_0)_{ii}\neq (\mathscr{D}_0)_{jj}.$ Then there exist two open sets $U\subset\mathrm{O}_d(\RR)\times\D_d(\RR)$ containing $(P_0,\mathscr{D}_0)$ and $V\subset\mathcal{S}_d(\RR)$ such that $F_{|U}:U\longrightarrow V$ is a $C^\infty$-diffeomorphism.
\end{lemma}

\begin{proof*}
Let $G:\M_d(\RR)\times\M_d(\RR)\longrightarrow\M_d(\RR)$ be defined by $G(P,\mathscr{D})=P\mathscr{D}P^\top$. The function $G$ is of class $C^{\infty}$  and, for every $(M,\Delta)\in\M_d(\RR)\times\M_d(\RR),$
\begin{gather*}
DG(P_0,\mathscr{D}_0)(M,\Delta)=M\mathscr{D}_0P_0^\top+P_0\mathscr{D}_0M^\top+P_0\Delta P_0^\top.
\end{gather*}
It follows that $F=G_{|\mathrm{O}_d(\RR)\times\D_d(\RR)}$ is differentiable at $(P_0,\mathscr{D}_0).$ Take an element $H$ of the tangent space $T_{P_0}\mathrm{O}_d(\RR)$ at $P_0.$ Then $H=P_0A$ for some antisymmetric matrix $A.$ Let $\Delta\in\D_d(\RR).$ From above, we get
\begin{gather*}
DF(P_0,\mathscr{D}_0)(P_0A,\Delta)=P_0(A\mathscr{D}_0-\mathscr{D}_0A+\Delta)P_0^\top.
\end{gather*}
If $DF(P_0,\mathscr{D}_0)(P_0A,\Delta)=0$ then $A\mathscr{D}_0-\mathscr{D}_0A+\Delta=0$. Since, for any $i\neq j$,
$$
\begin{array}{rcl}
	  (A\mathscr{D}_0-\mathscr{D}_0A+\Delta)_{ij}&=&(A\mathscr{D}_0-\mathscr{D}_0A)_{ij}\\[8pt]
	  \; &=& \sum_{k=1}^d\big(A_{ik}(\mathscr{D}_0)_{kj}-(\mathscr{D}_0)_{ik}A_{kj}\big)\\[8pt]
    \; &=&	 A_{ij}(\mathscr{D}_0)_{jj}-(\mathscr{D}_0)_{ii}A_{ij}\\[8pt]
    \; &=&A_{ij}\big((\mathscr{D}_0)_{jj}-(\mathscr{D}_0)_{ii}\big)
    \end{array}
$$
and $(\mathscr{D}_0)_{jj}\neq(\mathscr{D}_0)_{ii},$ it follows that $A_{ij}=0.$ But $A$ is antisymmetric so that $A_{ii}=0.$ Hence $A=0.$ This yields, $0=A\mathscr{D}_0-\mathscr{D}_0A+\Delta=\Delta.$ Hence $DF(P_0,\mathscr{D}_0)$ is injective and since $\dim\big(\mathrm{O}_d(\RR)\times\D_d(\RR)\big)=\frac{d(d-1)}2+d=\frac{d(d+1)}2=\dim\mathcal{S}_d(\RR),$ $DF(P_0,\mathscr{D}_0)$ is an isomorphism. We conclude with the inverse function theorem.
\medskip
\end{proof*}

\begin{corollary}[Parametric diagonalization of Hessian]
\label{cordiff} 
Under \Cref{ass:H}, let $(P^*,\mathscr{D}^*)\in\mathrm{O}_d(\RR)\times\D_d(\RR)$ be such that $\nabla^2f(x^*)= P^*\mathscr{D}^* (P^*)^\top$.  Then there exists  $\eta>0$, and a unique couple $(P,\mathscr{D})\in C^1\big(B(x^*,\eta);\mathrm{O}_d(\RR)\big)\times C^1\big(B(x^*,\eta);\D_d(\RR)\big)$ such that $\big(P(x^*),\mathscr{D}(x^*)\big)=(P^*,\mathscr{D}^*)$  and 
\be\label{diagonalisation_second_derivee}
\nabla^2f(x)= P(x) \mathscr{D}(x) P(x)^\top,\quad\text{for all $x\in B(x^*,\eta)$.}
\ee
\end{corollary}

\begin{proof*}
Let $F$ be defined as in Lemma \ref{lemdiff}. Let $U\ni(P^*,\mathscr{D}^*)$ and $V$ be given by Lemma \ref{lemdiff}. Since $F(P^*,\mathscr{D}^*)=P^* \mathscr{D}^* (P^*)^\top=\nabla^2f(x^*)\in V$ and $f\in C^3(\RR^d;\RR)$, we may choose $\eta>0$ small enough to have $\nabla^2f(x)\in V$, for any $x\in B(x^*,\eta).$ By Lemma \ref{lemdiff}, $F$ is a $C^\infty$-diffeomorphism from $U$ onto $V.$ We infer that for every $x\in B(x^*,\eta),$ there exist unique $P(x)\in\mathrm{O}_d$ and $\mathscr{D}(x)\in\D_d(\RR)$ with $\big(P(x),\mathscr{D}(x)\big)\in U$ such that
\begin{gather*}
\nabla^2f(x) = F\big(P(x),\mathscr{D}(x)\big)=P(x) \mathscr{D}(x) P(x)^\top.
\end{gather*}
In particular, $\big(P(x^*),\mathscr{D}(x^*)\big)=(P^*,\mathscr{D}^*)$ and for any $x\in B(x^*,\eta)$, $\big(P(x),\mathscr{D}(x)\big)=F^{-1}\big(\nabla^2f(x)\big).$ Since $f\in C^3(\RR^d;\RR)$ and $F^{-1}$ is $C^\infty,$ it follows that $P$ and $\mathscr{D}$ are $C^1$ on $B(x^*,\eta).$
\medskip
\end{proof*}

\subsection{Some proofs}
\label{subsec:some_proof}

{\bf Proof of \Cref{rem:alignement-rate-cont}.} For $x(t)=(x_1(t),x_2(t))$ the gradient flow reads
\begin{equation}
\dot x_1=\left(-\lambda_1-2a x_2\right)x_1,
\qquad \dot x_2=-\lambda_2 x_2-a x_1^2 .
\label{eq:eqns_premiere_remarque}
\end{equation}

From the first equation and $x(0)=(\alpha_1,\alpha_2)$, we have
\[
x_1(t)=\alpha_{1}\exp\!\left(-\lambda_1 t-2a\int_0^t x_2(s)\,\dd s\right).
\]
By Theorem~\ref{HG_cas_continu}, if $|x(0)|=|(\alpha_1,\alpha_2)|$ is small enough, then the trajectory converges exponentially to $0$, and hence the function $x_2$ is integrable on $[0,\infty)$. Therefore, $L:=\int_0^{\infty}x_2(s)\,\dd s$ is well defined. Setting \(
C:=x_1(0)\exp(-2aL)>0,
\) we may rewrite the above identity as
$$
x_1(t)
=
C\,\exp(-\lambda_1 t)\exp\!\left(2a\int_t^{\infty}x_2(s)\,\dd s\right).
$$
Substituting this expression into the second equation in~\eqref{eq:eqns_premiere_remarque} yields
$$
\dot x_2+\lambda_2x_2=-aC^2\exp(-2\lambda_1 t)\exp\!\left(4a\int_t^{\infty}x_2(s)\,\dd s\right).
$$
By variation of constants,
$$
x_2(t)=\exp(-\lambda_2 t)\alpha_{2}-aC^2\int_0^t\exp(-\lambda_2(t-s))\exp(-2\lambda_1 s)\exp\!\left(4a\int_s^{\infty}x_2(u)\,\dd u\right)\,\dd s.
$$
We split this as
$$
x_2(t)=\exp(-\lambda_2 t)\alpha_{2}-aC^2\int_0^t\exp(-\lambda_2(t-s))\exp(-2\lambda_1 s)\,\dd s+R(t),
$$
where
\begin{equation}
R(t):=-aC^2\int_0^t\exp(-\lambda_2(t-s))\exp(-2\lambda_1 s)\left[\exp\!\left(4a\int_s^{\infty}x_2(u)\,\dd u\right)-1\right]\dd s.
\label{eq:def_R_t}
\end{equation}
Since $\lambda_2>2\lambda_1$, we have
$$
\int_0^t\exp(-\lambda_2(t-s))\exp(-2\lambda_1 s)\,\dd s=\frac{\exp(-2\lambda_1 t)-\exp(-\lambda_2 t)}{\lambda_2-2\lambda_1}.
$$
Hence $\displaystyle x_2(t)=-\frac{aC^2}{\lambda_2-2\lambda_1}\exp(-2\lambda_1 t)+\widetilde C\,\exp(-\lambda_2 t)+R(t)$
for some constant $\widetilde C\in\RR$. It remains to estimate $R(t)$. Since $\int_{0}^{\infty}|x_{2}(u)|\dd u<\infty$, by the mean value theorem,  
$$
\left|\exp\left(4a\int_{s}^{\infty}x_2(u)\dd u\right)-1\right|=\mathrm O\left(\left|\int_{s}^{\infty}x_{2}(u)\dd u\right|\right)=\mathrm O(\exp\left(-\lambda_{1}s\right))
$$
and hence, there exists $\widehat{C}>0$ such that  
$$
|R(t)|\leq \widehat{C}\exp\left(-\lambda_{2}t\right)\int_{0}^{t}\exp\left(\left(\lambda_{2}-3\lambda_{1}\right)s\right)\dd s=\mathrm o\left(\exp\left(-2\lambda_{1}t\right)\right).
$$

Thus, we have
$$
x_2(t)=-\frac{aC^2}{\lambda_2-2\lambda_1}\exp(-2\lambda_1 t)+\mathrm o(\exp(-2\lambda_1 t))
$$
and, consequently,
\begin{equation}
\frac{x_2(t)}{x_1(t)}=-\frac{aC}{\lambda_2-2\lambda_1}\exp(-\lambda_1 t)+\mathrm o(\exp(-\lambda_1 t)).
\label{eq:expansion_x_2_x_1}
\end{equation}
Finally, since $x_{2}(t)/x_{1}(t)\to 0$ as $t\to\infty$,~\eqref{eq:expansion_x_2_x_1} implies  
$$
\begin{aligned}
\dist\!\left(\rho(x(t)),T_{0}\extun\right)&=\dfrac{|x_{2}(t)|}{|x(t)|}=\dfrac{|x_{2}(t)|}{x_{1}(t)}+\mathrm{o}\left(\dfrac{|x_{2}(t)|}{x_{1}(t)}\right)\\
&= \frac{aC}{\lambda_2-2\lambda_1}\exp(-\lambda_1 t)+\mathrm o(\exp(-\lambda_1 t)),
\end{aligned}
$$
which shows~\eqref{eq:estimation_rho_x_t}.
\qed

{\bf Proof of~\Cref{rem:alignment_speed_sharpness_discrete}.}
Set $\mu_{i}=1-\gamma\lambda_{i}\in (0,1)$ for $i=1,\,2$ and notice that, since $\lambda_{2}>2\lambda_{1}$, one has $\m_{1}^{2}>\mu_{2}$. System~\eqref{eq:sys_discrete_eqns} takes the form 
\begin{equation}
\begin{aligned}
x_{1,k+1}
&=
\left(\mu_1 
-
2a\gamma x_{2,k}\right)x_{1,k},
\\
x_{2,k+1}
&=
\mu_2 x_{2,k}
-
a\gamma x_{1,k}^2.
\end{aligned}
\label{eq:sys_discrete_eqns_Appendix}
\end{equation}
On the one hand, let us show the existence of $\bar{\alpha}_{1}\neq 0$ such that, as $k\to\infty$, 
\begin{equation}
x_{1,k}=\bar{\alpha}_{1}\mu_{1}^{k}+\mathrm o\left(\mu_{1}^{k}\right).
\label{eq:expansion_x_1_k}
\end{equation}
Indeed, by the first equation in~\eqref{eq:sys_discrete_eqns}, for every $k\geq 1$, one has 
$$
x_{1,k}=\alpha_{1}\mu_{1}^{k}\prod_{i=0}^{k-1}\left(1
-2a\gamma x_{2,i}/\mu_{1}\right).
$$
From~\eqref{eq:Phi_exp_conv_discret}, if $|x_0|$ is small enough, there exists $c>0$ such that $|x_{2,i}|\leq c\mu_{1}^{i}$ for all $i\geq 0$, and hence $\sum_{i=0}^{\infty}|x_{2,i}|<\infty$, which implies that 
$$
0<P=\prod_{i=0}^{\infty}\left(1
-2a\gamma x_{2,i}/\mu_{1}\right)<\infty.
$$
In turn,~\eqref{eq:expansion_x_1_k} holds with $\bar{\alpha}:=\alpha_{1} P>0$. On the other hand, let us show that, as $k\to\infty$, 
\begin{equation}
x_{2,k}=-\frac{a\gamma\bar\alpha_1^2}{\mu_1^2-\mu_2}\mu_1^{2k}+\mathrm o(\mu_1^{2k}). 
\label{eq:expansion_x_2_k}
\end{equation}
Setting $y_{k}=x_{2,k}/\mu_{1}^{2k}$,~\eqref{eq:sys_discrete_eqns_Appendix}, and~\eqref{eq:expansion_x_1_k} yield 
$$
(\forall\,k\geq 0)\quad y_{k+1}= \frac{\mu_2}{\mu_{1}^{2}}y_{k}-\frac{a\gamma x_{1,k}^{2}}{\mu_{1}^{2(k+1)}}= q y_{k}+b+\eps_{k}, 
% \frac{\mu_2}{\mu_{1}^{2}} y_{k}-\frac{a\gamma\bar{\alpha}_{1}^{2}}{\mu_{1}^{2}}+\eps_k,
$$
where $q= \frac{\mu_2}{\mu_{1}^{2}}\in(0,1)$, $b=- a\gamma\bar{\alpha}_{1}^{2}/\mu_{1}^{2}$, and $\eps_{k}\to 0$. Letting  $\bar{y}=b/(1-q)=-a\gamma\bar{\alpha}_{1}^{2}/\left(\mu_{1}^{2}-\mu_{2}\right)$ and $\eta_k=y_k-\bar{y}$, we have 
\begin{equation}
(\forall\,k\geq 0)\quad \eta_{k+1}=q\eta_{k}+\eps_{k}=q^{k+1}\eta_{0}+\sum_{i=0}^{k}q^{k-i}\eps_i. 
% \frac{\mu_2}{\mu_{1}^{2}} y_{k}-\frac{a\gamma\bar{\alpha}_{1}^{2}}{\mu_{1}^{2}}+\eps_k,
\label{eq:eta_k_recurrence}
\end{equation}
Let $\delta>0$ and pick $N\in\NN$ such that $|\eps_i|\leq\delta$ for all $i\geq N$. Then, for $k>N$, 
\begin{multline*}
\left|\sum_{i=0}^{k}q^{k-i}\eps_i\right|=\left| \sum_{i=0}^{N-1}q^{k-i}\eps_{i}+\sum_{i=N}^{k}q^{k-i}\eps_{i}\right|\\
\leq q^{k}\left| \sum_{i=0}^{N-1}q^{-i}\eps_{i}\right|+\delta\sum_{i=N}^{k}q^{k-i}\leq q^{k}\left| \sum_{i=0}^{N-1}q^{-i}\eps_{i}\right|+\delta \sum_{i=0}^{\infty}q^{i}.
\end{multline*}
Using that $q\in(0,1)$, we get that $\limsup_{k\to\infty}\left|\sum_{i=0}^{k}q^{k-i}\eps_i\right|\leq\delta \sum_{i=0}^{\infty}q^{i},$
and, since $\sum_{i=0}^{\infty}q^{i}<\infty$ and $\delta$ is arbitrary, we deduce that $\sum_{i=0}^{k}q^{k-i}\eps_i\to 0$ and, hence, by~\eqref{eq:eta_k_recurrence}, $\eta_k\to 0$, which implies that 
$$
x_{2,k}=\bar{y}\mu_{1}^{2k}+\mathrm o(\mu_{1}^{2k}),
$$
from which~\eqref{eq:expansion_x_2_k} follows.

Finally, it follows from~\eqref{eq:expansion_x_1_k} and~\eqref{eq:expansion_x_2_k} that
$$
x_k=\left(\bar{\alpha}_1\mu_1^k+ \mathrm o(\mu_1^k)\right)e_1+\left(-\frac{a\gamma \bar{\alpha}_1^2}{\mu_1^2-\mu_2}\mu_1^{2k}+\mathrm o(\mu_1^{2k})\right) e_2,
$$
which yields $|x_k|=|\bar{\alpha}_1|\,\mu_1^k+\mathrm o(\mu_1^k)$, and hence 
$$
\mathrm{dist}\!\left(\rho(x_k),\,T_0\extun\right)=\dfrac{|x_{2,k}|}{|x_k|}=\frac{|a|\gamma \alpha_1^2}{|\bar{\alpha}_1|(\mu_1^2-\mu_2)}\mu_1^{k}+\mathrm o(\mu_1^{k}),
$$
which shows~\eqref{eq:dist_talweg_ex_2}.
\qed

% ---------------------------------

% ---------------------------------

% ---------------------------------

% ---------------------------------

%%%%%%%%%%%%%%%%%%%%%%%%%%%%%%%%%%%%%%%%%%%%%%%%%%%%%%%

\end{document}